\documentclass[a4paper,reqno,12pt]{article}

\usepackage{amsmath,amssymb,amsthm,mathtools,bm}
\usepackage{geometry}
\usepackage{microtype}
\usepackage{enumitem}
\usepackage{booktabs}
\usepackage{array}
\usepackage{aliascnt}
\usepackage{fancyhdr}
\usepackage[numbers,sort&compress]{natbib}
\usepackage{indentfirst}
\usepackage{hyperref}
\usepackage[nameinlink,noabbrev]{cleveref}
\allowdisplaybreaks[3]
\numberwithin{equation}{section}
\setlist{itemsep=0.2em,topsep=0.35em}

\makeatletter
\renewcommand\section{\@startsection{section}{1}{\z@}%
  {-2.2ex \@plus -0.5ex \@minus -0.2ex}%
  {1.1ex \@plus 0.2ex}%
  {\normalfont\large\bfseries\raggedright}}
\makeatother

\newtheorem{theorem}{Theorem}[section]
\newaliascnt{lemma}{theorem}
\newtheorem{lemma}[lemma]{Lemma}
\aliascntresetthe{lemma}
\newaliascnt{proposition}{theorem}
\newtheorem{proposition}[proposition]{Proposition}
\aliascntresetthe{proposition}
\newaliascnt{corollary}{theorem}

\aliascntresetthe{corollary}
\theoremstyle{remark}
\newaliascnt{remark}{theorem}
\newtheorem{remark}[remark]{Remark}
\aliascntresetthe{remark}
\crefname{theorem}{theorem}{theorems}
\Crefname{theorem}{Theorem}{Theorems}
\crefname{lemma}{lemma}{lemmas}
\Crefname{lemma}{Lemma}{Lemmas}
\crefname{proposition}{proposition}{propositions}
\Crefname{proposition}{Proposition}{Propositions}
\crefname{corollary}{corollary}{corollaries}
\Crefname{corollary}{Corollary}{Corollaries}
\crefname{remark}{remark}{remarks}
\Crefname{remark}{Remark}{Remarks}
\crefname{equation}{equation}{equations}
\Crefname{equation}{Equation}{Equations}
\crefname{section}{section}{sections}
\Crefname{section}{Section}{Sections}

\newcommand{\dd}{\,\mathrm d}
\newcommand{\norm}[1]{\left\lVert #1\right\rVert}

\newcommand{\doi}[1]{\href{https://doi.org/#1}{doi: #1}}

\title{\bfseries Global Well-Posedness for Planar Compressible Magnetohydrodynamics with Degenerate Viscosity on Unbounded Domains}
\author{Shixuan Li$^a$, Hongyu Wang$^b$ \thanks{
		Email addresses: {lsx981017@163.com} (S. X. Li), 
        {hyuwang\_a@163.com} (H. Y. Wang).} \\[3mm]  a. School of Mathematics and Computer Science,\\ Nanchang University,  Nanchang 330031, P. R. China; \\
	b. School of Mathematics and Computer Science, \\Nanchang University,  Nanchang 330031, P. R. China.}
\date{}
\begin{document}

\maketitle

%We study the global existence and large-time behavior of strong
%solutions to the one-dimensional planar compressible
%magnetohydrodynamic system in unbounded domains. %provided that the longitudinal viscosity and heat conductivity are given by $\mu(v)=v^{-\alpha}$ and $\kappa(\theta)=\theta^\beta$.

%We consider the Cauchy problem and thermally insulated half-line problems with either homogeneous Dirichlet or Neumann conditions on the magnetic field.

%For every fixed $\beta\geq0$ and arbitrarily large $H^1$ initial
%perturbations of a constant equilibrium with the initial specific
%volume and temperature bounded away from zero, we establish
%the existence of a unique global strong solution, provided that
%the nonnegative viscosity exponent $\alpha$ is sufficiently small in terms of the initial data.
%Both the specific volume and the temperature remain bounded above
%and away from zero uniformly in space and time. Moreover, for each fixed admissible $\alpha$, the solution converges to equilibrium uniformly in space as time tends to infinity. The result allows the longitudinal viscosity to degenerate as the density approaches vacuum, without requiring a positive constant part or any smallness condition on the initial perturbation.

\begin{abstract}The one-dimensional planar compressible magnetohydrodynamic system is studied on unbounded domains. The longitudinal viscosity obeys a power law of the specific volume, which permits it to degenerate as the density approaches vacuum, and the heat conductivity follows a power law of the temperature. For any fixed non-negative exponent in the heat conductivity term, the existence and uniqueness of global strong solutions are established for arbitrarily large initial data, provided that the positive exponent for viscosity is sufficiently small. The specific volume and temperature possess time-independent upper and lower bounds. Moreover, for each admissible viscosity exponent, the solution converges uniformly in space to the equilibrium state as time tends to infinity.
\end{abstract}

\noindent\textbf{Keywords:} planar compressible magnetohydrodynamics;
degenerate viscosity; large initial data; uniform-in-time estimates;
large-time behavior.

\section{Introduction}\label{sec:intro}
\indent We consider the one-dimensional planar compressible  magnetohydrodynamic
(MHD) system with density-dependent longitudinal viscosity and
temperature-dependent heat conductivity which read as follows \cite{Cabannes1970, JeffreyTaniuti1964, LandauLifshitz1984}:
\begin{align}
	v_t & =u_x  \label{eq:mass}\\[3mm]
	u_t+\left(P+\frac{1}{2}|\mathbf{b}|^2\right)_x & =\left(\frac{\mu u_x}{v}\right)_x,  \label{eq:mom}\\[3mm]
	\mathbf{w}_t-\mathbf{b}_x & =\left(\frac{\lambda \mathbf{w}_x}{v}\right)_x,  \label{eq:trans}\\[3mm]
	(v \mathbf{b})_t-\mathbf{w}_x & =\left(\frac{\nu \mathbf{b}_x}{v}\right)_x, \label{eq:ind}
\end{align}
 \begin{align}
	 \biggl(e&+\frac{u^2+|\mathbf{w}|^2+v|\mathbf{b}|^2}{2}\biggr)_t+\left\{u\left(P+\frac{1}{2}|\mathbf{b}|^2\right)-\mathbf{w} \cdot \mathbf{b}\right\}_x \nonumber\\
	& =\left\{\frac{\mu u u_x}{v}+\frac{\lambda \mathbf{w} \cdot \mathbf{w}_x}{v}+\frac{\nu \mathbf{b} \cdot \mathbf{b}_x}{v}+\frac{\kappa \theta_x}{v}\right\}_x, \label{eq:temp}
\end{align}where  $t>0$ is time, $x\in\Omega\subset\mathbb{R}=(-\infty,+\infty)$ denotes the Lagrange mass coordinate.  The
unknown functions 
 $v>0$, $u$, $\mathbf{w}\in\mathbb{R}^2$,  $\mathbf{b}\in\mathbb{R}^2$,  $e>0$,  $\theta>0$,   and $P$ are,  respectively, 
the specific volume of  the gas,  longitudinal velocity, transverse velocity, transverse magnetic field,
 internal energy,  absolute temperature, and pressure.  We will concentrate on the ideal, polytropic gas: 
\begin{equation}
	P=\frac{R \theta}{v}, \quad e=c_v \theta+\text { const.},
\end{equation}
where $R>0$ is the gas constant and $c_v > 0$ is the specific heat at constant volume. The transverse
shear viscosity coefficient $\lambda> 0$ and the magnetic diffusivity $\nu>0$ are constants. The
longitudinal viscosity  coefficient $\mu$  and heat conductivity  coefficient $\kappa$ satisfy
\begin{equation}\label{eq:viscosity}
	\mu=\mu(v)= \widetilde{\mu} v^{-\alpha},\qquad \kappa=\kappa(\theta)=\widetilde\kappa\theta^\beta,
\end{equation}
with constants $ \widetilde\mu \geq0,   \widetilde\kappa>0,$ and $\alpha\geq0, \beta\geq0. $  
Since the density is $1/v$,  for $\alpha>0$, the viscosity   tends to
zero as the density approaches vacuum.  For simplicity, let $R=c_v=\lambda=\nu=\widetilde\mu_2=\widetilde\kappa=1$.
  
The system  \eqref{eq:mass}-\eqref{eq:viscosity} is imposed with the initial conditions: 
\begin{equation}\label{eq:data}
 (v,u,\theta,\mathbf{w},\mathbf{b})(x,0)=(v_0,u_0,\theta_0,\mathbf{w}_0,\mathbf{b}_0)(x),
 \qquad x\in\Omega,
\end{equation}
and three types of far-ﬁeld and boundary conditions:
\begin{enumerate}
	\item[(1)]  Cauchy problem
	\begin{equation}
		\Omega=\mathbb{R},\quad\lim _{|x| \rightarrow\infty}(v, u, \theta, \mathbf{w}, \mathbf{b})(x,t)=(1,0,1,0,0), \quad t>0;\label{eq:bc-C}
	\end{equation}
	\item[(2)]The half-line problem with a homogeneous Neumann condition on the
	magnetic field
\begin{equation}
	\begin{gathered}
	\Omega=(0, \infty),\quad	u(0, t)=\mathbf{w}(0,t)=0, \quad \theta_x(0, t)=\mathbf{b}_x(0, t)=0,\\ \lim _{x \rightarrow \infty}(v, u, \theta, \mathbf{w},\mathbf{b})(x, t)=(1,0,1,0,0), \quad t>0;\label{eq:bc-N}
	\end{gathered}
\end{equation}	
	\item[(3)]The half-line problem with a homogeneous Dirichlet condition on the
	magnetic field
	\begin{equation}\label{eq:bc-D}
	\begin{gathered}
			\Omega=(0, \infty),\quad	u(0, t)=\mathbf{w}(0, t)=\mathbf{b}(0, t)=0, \quad\theta_x(0, t)=0, \\\lim _{x \rightarrow \infty}(v, u, \theta, \mathbf{w},\mathbf{b})=(1,0,1,0,0), \quad t>0.
	\end{gathered}\end{equation}
\end{enumerate}

The large-data theory for one-dimensional planar MHD on bounded domains goes back to
Kazhikhov~\cite{Kazhikhov1987}, who obtained global a priori estimates for constant
transport coefficients. Chen and Wang~\cite{ChenWang2002} studied a free-boundary
problem with general constitutive laws, while Wang~\cite{Wang2003} considered the
initial-boundary value problem on a fixed bounded interval with variable transport
coefficients. For constant viscosities and magnetic diffusivity and $\kappa=\widetilde{\kappa}\theta^\beta, \beta>0,$ Hu and Ju~\cite{HuJu2015} proved the existence and uniqueness of global strong
solutions for large initial data. Huang, Shi, and Sun~\cite{HuangShiSun2019}
further treated $\mu(v)=\widetilde{\mu}_1+\widetilde{\mu}_2v^{-\alpha}$, where $\widetilde{\mu}_1>0,
\widetilde{\mu}_2\geq0, \alpha\geq0,$
and established unique global strong solutions for large $H^1$ data for every fixed
$\alpha\geq0$. For constant viscosities and heat conductivity, Li and
Shang~\cite{LiShang2019} proved unique global strong solutions for a class of
magnetic diffusivities depending on the specific volume. For constant viscosities
and magnetic diffusivity, Huang, Shi, and Sun~\cite{HuangShiSun2021} established
uniform-in-time positive lower and upper bounds for $v$ and $\theta$, as well as
nonlinear exponential stability, for every $\beta\geq0$. More recently, Song and
Zhao~\cite{SongZhao2025} considered
\[
\mu(\theta),\lambda(\theta),\nu(\theta)\sim\theta^\alpha,
\qquad
\kappa(\theta)\sim\theta^\beta,
\qquad \beta>0,
\]
and proved unique global strong solutions and exponential stability for large $H^2$
data when $\alpha\geq0$ is sufficiently small.

On unbounded domains, global solvability and uniform-in-time estimates are distinct
issues. For constant viscosities and heat conductivity, L\"u, Shi, and
Xiong~\cite{LuShiXiong2021} proved the global existence of strong solutions for
large initial data. Cao, Peng, and Sun~\cite{CaoPengSun2021} subsequently considered
\[
\mu(v)=\widetilde{\mu}_1+\widetilde{\mu}_2v^{-\alpha},
\quad \mu_1>0,\quad \mu_2 \geq 0, \quad
\alpha\geq0,\qquad
\kappa(\theta)=\widetilde{\kappa}\theta^\beta,
\quad \beta>0,
\]
and established the existence and uniqueness of global strong solutions to the
Cauchy problem and the half-line problem with homogeneous Dirichlet conditions on
the magnetic field. For constant longitudinal viscosity and
$\kappa(\theta)=\widetilde{\kappa}\theta^\beta$, $\beta\geq0$, Tong, Wang, and
Zhang~\cite{TongWangZhang2026} proved the existence and uniqueness of global strong
solutions to the half-line problem with homogeneous Neumann conditions on the
magnetic field. For any fixed $T>0$, these results provide positive lower and upper
bounds for the specific volume and temperature on $[0,T]$, with constants that may
depend on $T$.

The large-time theory on unbounded domains was first developed for the
one-dimensional compressible Navier--Stokes equations.
Jiang~\cite{Jiang1999,Jiang2002} firstly obtained uniform-in-time positive lower and upper
bounds for the specific volume. Later on, Li and Liang~\cite{LiLiang2016}   established uniform-in-time estimates
and convergence to equilibrium for large initial data with constant viscosity and
heat conductivity, while Li and Xu~\cite{LiXu2026} extended these results to the
case of constant viscosity and temperature-dependent heat conductivity with
$\beta>0$.

For the case $\mu(\theta)=\theta^\alpha, 
\kappa(\theta)=\theta^\beta$,
Dong and Guo~\cite{DongGuo2025} established global existence and
asymptotic stability for the Cauchy problem with large
$H^2\times H^1\times H^1$ initial data when $\beta\geq0$ is fixed
and $\alpha\geq0$ is sufficiently small. Very recently, Wang and
Zhang~\cite{WangZhang2026} lowered the regularity requirement to
$H^1\times H^1\times H^1$  for the Cauchy problem on $\mathbb R$ and two initial-boundary
value problems on $\mathbb R_+$.

For planar MHD, the large-time behavior of strong solutions on unbounded domains
has also been studied under various assumptions on the transport coefficients.
Shang~\cite{Shang2022} considered viscosities and heat conductivity with a common
factor $h(v)\theta^a$, subject to suitable growth conditions on $h$, and established global classical solutions and asymptotic stability for general large initial data when $|a|$ is
sufficiently small and the magnetic diffusivity is sufficiently large.

More recently, Li, Liao, and Wan~\cite{LiLiaoWan2026} studied the MHD Cauchy
problem with constant longitudinal viscosity and
\[
\kappa(v,\theta)=\kappa_1+\kappa_2v\theta^b,
\qquad
\kappa_1,\kappa_2>0,\quad b>\frac32,
\]
and established uniform-in-time positive lower and upper bounds for the specific
volume and temperature, together with convergence to equilibrium. Their result
also applies to the half-line problem with homogeneous Dirichlet conditions on
the magnetic field.

These works established large-data asymptotic stability for planar MHD under
several classes of transport coefficients. A recent breakthrough is due to  Li and Liu~\cite{LiLiu2026}, who established
global strong solutions, uniform-in-time bounds, and convergence to equilibrium
for arbitrarily large $H^1$ data on the whole and half lines. In particular, they considered the
two viscosity laws
\[
\mu=\mu(v)=\mu_1+\mu_2v^{-\alpha},
\qquad \alpha\geq0,
\]
and
\[
\mu=\mu(\theta)=\mu_0\theta^\gamma,
\qquad \gamma\geq0.
\]  However, their methods rely crucially on the uniform positive lower bound of the viscosity coefficient \(\mu\), which fails to apply to the degenerate case \eqref{eq:viscosity}. In the present paper, we consider the genuinely degenerate viscosity law \eqref{eq:viscosity}, where the longitudinal viscosity degenerates to zero as the specific volume \(v\to\infty\). Our aim is to
establish the corresponding large-data global existence, uniform-in-time estimates,
and large-time stability for sufficiently small $\alpha>0$.

We first state the local existence   result for the viscosity
coefficient in \eqref{eq:viscosity} whose proof can be found in \cite{CaoPengSun2021,LuShiXiong2021,TongWangZhang2026, LiLiu2026}.
\begin{lemma}[\cite{CaoPengSun2021,LuShiXiong2021,TongWangZhang2026, LiLiu2026}]\label{lem:01a}
	Suppose that \begin{equation}\label{}
	0\leq\alpha\leq1,\quad	\beta\geq0,
	\end{equation}
	and that the initial data $(v_0, u_0, \theta_0, \mathbf{w}_0, \mathbf{b}_0)$ satisfies
	\begin{equation}\label{eq:init}
		(v_0-1,u_0,\theta_0-1,\mathbf{w}_0,\mathbf{b}_0)\in H^1(\Omega),
		\qquad \inf\limits_{\Omega} v_0>0, \quad\inf\limits_{\Omega} \theta_0>0,
	\end{equation} and are compatible with  \eqref{eq:bc-N}, \eqref{eq:bc-D}.  Then there exists some $T_0>0$  such that   the initial-boundary-value problem \eqref{eq:mass}-\eqref{eq:bc-C},
or \eqref{eq:mass}-\eqref{eq:data}, \eqref{eq:bc-N}, or  \eqref{eq:mass}-\eqref{eq:data}, \eqref{eq:bc-D}
has a unique global strong solution $(v,u,\theta,\mathbf{w},\mathbf{b})$ satisfying 
\begin{equation}\label{eq:finite}
 \begin{gathered}
 v-1,u,\theta-1,\mathbf{w},\mathbf{b}\in C([0,T_0];H^1(\Omega)),\\[2mm]
 v_t\in C([0,T_0];L^2(\Omega))\cap L^2(0,T_0;H^1(\Omega)),\\[2mm]
 u_t,\mathbf{w}_t,\mathbf{b}_t,\theta_t,u_{xx},\mathbf{w}_{xx},\mathbf{b}_{xx},\theta_{xx}
 \in L^2(\Omega\times(0,T_0)).
 \end{gathered}
\end{equation}
The specific volume and temperature are bounded above and away from zero on
this interval. 
For a fixed initial $H^1$ norm and fixed positive lower bounds for $v_0$ and
$\theta_0$, $T_0$ can be chosen independently of $\alpha\in[0,1]$.
\end{lemma}

Our main result extends this local solution to the whole interval $[0,\infty)$   and gives  estimates
independent of the final time.
\begin{theorem}\label{thm:main} 
Under the assumptions of Lemma \ref{lem:01a}, there exists a positive constant $\alpha_0\in(0,1]$, 
 depending only on $\beta$,  $\|(v_0-1, u_0, \theta_0-1, \mathbf{w}_0, \mathbf{b}_0)\|_{H^1(\Omega)},$ $\inf\limits_{\Omega} v_0,$ and $\inf\limits_{\Omega} \theta_0$,  such that  the initial-boundary-value problem \eqref{eq:mass}-\eqref{eq:bc-C},
or \eqref{eq:mass}-\eqref{eq:data}, \eqref{eq:bc-N}, or  \eqref{eq:mass}-\eqref{eq:data}, \eqref{eq:bc-D} with $\alpha\in (0,\alpha_0]$
admits  a unique global strong solution $(v,u,\theta,\mathbf{w},\mathbf{b})$  satisfying 
\begin{equation}\label{eq:global}
 \begin{gathered}
 v-1,u,\theta-1,\mathbf{w},\mathbf{b}\in C([0,T];H^1(\Omega)),\\[2mm]
 v_t\in C([0,T];L^2(\Omega))\cap L^2(0,T;H^1(\Omega)),\\[2mm]
 u_t,\mathbf{w}_t,\mathbf{b}_t,\theta_t,u_{xx},\mathbf{w}_{xx},\mathbf{b}_{xx},\theta_{xx}
 \in L^2(\Omega\times(0,T)).
 \end{gathered}
\end{equation}
 Moreover, there exist
constants $c,C>0$ depending only on $\beta$,  $\|(v_0-1, u_0, \theta_0-1, \mathbf{w}_0, \mathbf{b}_0)\|_{H^1(\Omega)},$ $\inf\limits_{\Omega} v_0,$ and $\inf\limits_{\Omega} \theta_0$,  such that
\begin{equation}\label{eq:mainvt}
 0<c\leq v(x,t),\theta(x,t)\leq C,
 \qquad x\in\Omega,\quad t\geq0,
\end{equation}
\begin{equation}\label{eq:mainH1}
 \sup_{t\geq0}
 \norm{(v-1,u,\theta-1,\mathbf{w},\mathbf{b})(\cdot,t)}_{H^1(\Omega)}^2\leq C,
\end{equation}
and
\begin{equation}\label{eq:mainds}
 \begin{aligned}
 \int_0^\infty\bigl\{&
 \norm{v_x(\cdot,t)}_{L^2(\Omega)}^2
 +\norm{(u_x,\theta_x,\mathbf{w}_x,\mathbf{b}_x)(\cdot,t)}_{H^1(\Omega)}^2\\[2mm]
 &+\norm{(u_t,\theta_t,\mathbf{w}_t,\mathbf{b}_t)(\cdot,t)}_{L^2(\Omega)}^2
 \bigr\}\dd t\leq C.
 \end{aligned}
\end{equation}
For every $2<p\leq\infty$,
\begin{equation}\label{eq:maingr}
 \begin{aligned}
 &\lim\limits_{t\rightarrow \infty}\norm{(v-1,u,\theta-1,\mathbf{w},\mathbf{b})(\cdot,t)}_{L^p(\Omega)}+\norm{(v_x,u_x,\theta_x,\mathbf{w}_x,\mathbf{b}_x)(\cdot,t)}_{L^2(\Omega)}
 =0.
 \end{aligned}
\end{equation}
\end{theorem}

\begin{remark}\label{rem:large-data}
No smallness assumption is imposed on the initial perturbation
$(v_0-1,u_0,\theta_0-1,\mathbf{w}_0,\mathbf{b}_0)$. For each fixed initial data satisfying
\eqref{eq:init}, there exists $\alpha_0>0$ such that the conclusions of
Theorem~\ref{thm:main} hold for all $0\leq\alpha\leq\alpha_0$. The heat-conductivity exponent
$\beta\geq0$ is arbitrary but fixed; in particular, the constant
heat-conductivity case $\beta=0$ is included. %To our knowledge, the large-data and density-degenerate  longitudinal viscosity asymptotic stability for planar compressible magnetohydrodynamic system has not been proved previously for the Cauchy and half-line problems.

\end{remark}

\begin{remark}\label{rem:degenerate-viscosity} Our case \(0<\alpha\leq\alpha_0\)  corresponds to a genuinely degenerate viscosity regime, where the viscosity coefficient admits no uniform positive lower bound and decays to zero as the specific volume tends to infinity. This setting fundamentally differs from the nondegenerate density-dependent viscosity models investigated in \cite{LiLiu2026}, where a positive constant viscosity term is present to prevent degeneracy.  
\end{remark}

We now outline the main difficulties and key estimates involved in the
proof. To extend the local strong solutions obtained in
Lemma~\ref{lem:01a} globally in time, we derive  time-independent a priori estimates in suitable higher norms. As in the related stability analyses
\cite{HuangShiSun2021,LiXu2026,LiLiu2026}, a central issue is to obtain  time-independent positive 
 upper and  lower bounds for the specific volume and
temperature. The possible degeneracy in longitudinal viscosity $\mu(v)=v^{-\alpha}$ and strong magnetic coupling effect  are  the intrinsic
difficulty, although $0<\alpha\leq 1$ may be very small.  

We first consider the bounds for the specific volume. Since $v_t=u_x$,
the longitudinal viscous flux can be expressed as a time derivative:
\[
 M_\alpha(v)=\int_1^v\frac{\mu(z)}{z}\dd z,
 \qquad
 \partial_t M_\alpha(v)=\frac{\mu(v)u_x}{v}.
\]
This identity underlies the local representation used by Li and
Liu \cite{LiLiu2026} to estimate $v$. For their viscosity law
$\mu(v)=\mu_1+\mu_2v^{-\alpha}$, with $\mu_1,\mu_2>0$, the corresponding
primitive contains the term $\mu_1\log v$ and grows without bound as
$v\to\infty$. In the present case, however,
\begin{equation*}
 M_\alpha(v)=\frac{1-v^{-\alpha}}{\alpha}
 \longrightarrow\frac1\alpha
 \qquad(v\to\infty,\ \alpha>0),
\end{equation*}
while $M_0(v)=\log v$. Thus, an upper bound for $M_\alpha(v)$ is not enough to
provide an upper bound for $v$.  To overcome this difficulty, we use a bootstrap involving only $v$.
Assume temporarily that $M^{-1}\leq v\leq M$, with $M\geq e$.
Under the smallness conditions on $\alpha$ in \eqref{eq:alpha-small},
we recover the comparison
\begin{equation*}
 |M_\alpha(v)-\log v|\leq1
 \quad\Longrightarrow\quad
 e^{-1}v\leq e^{M_\alpha(v)}\leq ev,
\end{equation*}
together with \(\|\mu-1\|_{L^\infty}\leq\delta_*.\);
see \eqref{eq:primitive-close}--\eqref{eq:fM}.
These comparisons allow us to adapt the normalized  argument
of Li and Liu \cite{LiLiu2026} to obtain a uniform positive lower
bound for $v$, and the inverse-temperature truncation argument
then gives a uniform positive lower bound for $\theta$.
In both arguments, the key observation is that the set of times
at which the temperature dissipation exceeds a fixed threshold
has uniformly bounded measure. The resulting lower bounds are
independent of $M$ and $T$, uniformly for the admissible $\alpha$,
and require no uniform estimate for $v_x$;
see Propositions~\ref{p:vlow} and~\ref{p:tlow}.

We next turn to the uniform upper bound for $v$, which requires
control of magnetic-pressure oscillations. Following
Li and Liu~\cite{LiLiu2026}, we obtain this control through a
time-integrated estimate of $v\theta|\mathbf{b}|^2$.
By \eqref{eq:vtb}, it suffices to estimate the time integral of
$\mathcal K_\beta(\theta)\eta$, where $\eta$ and $\mathcal K_\beta$
are defined in \eqref{eq:eta} and \eqref{eq:K-beta}.
To this end, we use the auxiliary elliptic problem
\[
 -\left(\frac{\psi_x}{v}\right)_x+\psi=\eta,
 \qquad F=-\frac{\psi_x}{v}.
\]
Lemma~\ref{l:res} provides bounds for $F$, $\psi$, and the energy
$\mathcal H$ in \eqref{eq:FH}, independently of $M$.

The additional difficulty in applying this argument lies in the
absence of a positive constant part in the viscosity.
In the density-dependent case of~\cite{LiLiu2026}, the positive constant $\mu_1$ produces a positive multiple of
$\int_\Omega\eta u^2\dd x$ in the elliptic energy identity,
which helps control the nonlinear terms.
For the present viscosity law, the bootstrap assumptions give
$\|\mu(v)-1\|_{L^\infty}\leq\delta_*$ by \eqref{eq:mu-close}.
We therefore write $\mu=1+(\mu-1)$ to retain the
constant-viscosity cancellation, leaving an additional error
involving $\mu-1$. This yields
\begin{equation}\label{eq:intro-pp-elliptic}
 \mathcal H_t
 +c\int_\Omega\mathcal K_\beta(\theta)\eta\dd x
 +c\int_\Omega\eta u^2\dd x
 \leq C+
 \left|\int_\Omega(\mu-1)Fuu_x\dd x\right|.
\end{equation}

The additional term must be controlled before a uniform estimate
for $v_x$ is available. Integrating it by parts would introduce
$\mu_x=\mu'(v)v_x$. We therefore estimate it directly by
Young's inequality:
\begin{equation}\label{eq:intro-pp-error}
 \begin{aligned}
 \left|\int_\Omega(\mu-1)Fuu_x\dd x\right|
 &\leq\varepsilon\int_\Omega\frac{\mu u_x^2}{v}\dd x
   +C\varepsilon^{-1}\delta_*^2
        \int_\Omega\frac v\mu F^2u^2\dd x\\
 &\leq\varepsilon\int_\Omega\frac{\mu u_x^2}{v}\dd x
   +C\varepsilon^{-1}\delta_*^2
        \left(1+\int_\Omega\eta u^2\dd x\right).
 \end{aligned}
\end{equation}

In the second inequality of \eqref{eq:intro-pp-error}, we use
$v\leq C(1+\Phi(v))$, together with the bounds for $F$ and $\mu$
and the basic energy estimate, to keep the constants independent
of $M$. It remains to control $\int_\Omega\frac{\mu u_x^2}{v}\dd x$, since the entropy
dissipation contains the additional factor $1/\theta$.
Combining the entropy estimate with the temperature equation
tested against $(1-2/\theta)_+$  gives
\[
 \int_s^t\!\int_\Omega\frac{\mu u_x^2}{v}\dd x\dd r
 \leq C_h\left(
 1+\int_s^t\!\int_\Omega
 \mathcal K_\beta(\theta)\eta\dd x\dd r\right);
\]
see \eqref{eq:heat-conv}.
After integrating \eqref{eq:intro-pp-elliptic} in time,
we first choose $\varepsilon$ sufficiently small to absorb
the contribution involving $\mathcal K_\beta(\theta)\eta$,
and then choose $\delta_*$ sufficiently small to absorb the
one involving $\eta u^2$.
This closes the elliptic estimate with constants independent
of $M$, before any uniform estimate for $v_x$ is used.

Proposition~\ref{p:bcap} consequently gives \eqref{eq:b-cap}.
Together with the magnetic dissipation and the Cauchy--Schwarz
inequality, this yields
\begin{equation}\label{eq:intro-pp-magnetic}
 \int_s^t\!\int_{I_N}|\mathbf{b}||\mathbf{b}_x|\dd x\dd r
 \leq C\sqrt{1+t-s}.
\end{equation}
We can then apply the normalized  argument
of~\cite{LiLiu2026}. Above a fixed threshold, the thermal pressure
provides a negative drift at good times, while the contribution
from the bad-time set has a uniformly bounded time integral.
This drift dominates the square-root growth in
\eqref{eq:intro-pp-magnetic}, giving $v\leq C$ independently
of $M$; see Proposition~\ref{p:v-up}.

To close the volume bootstrap, we fix $\delta_*$ by the preceding
absorption argument and then choose $M_*$ sufficiently large
that the range of $v_0$ and the improved bounds $[c_*,C_*]$
lie strictly inside $(2/M_*,M_*/2)$.
Only afterward do we choose $\alpha_0$ by
\eqref{eq:alpha-choice}.
For $0<\alpha\leq\alpha_0$, the resulting strict improvement
of the assumed volume bounds excludes a finite first exit.
Thus the bootstrap closes without assuming derivative bounds
or a uniform upper bound for $\theta$.

Once \eqref{eq:state} is established, the higher-order estimates
follow the framework of~\cite{LiLiu2026}.
The weighted mechanical and transverse flux estimates,
together with an additional temperature truncation when $\beta=0$,
provide the inputs for the estimate of $v_x$.
The subsequent longitudinal velocity and temperature estimates
yield the uniform upper bound for $\theta$ and the remaining
regularity needed for continuation.
These estimates give global existence and, for each fixed
admissible $\alpha$, convergence to equilibrium;
see Sections~\ref{sec:high} and~\ref{sec:full}. The rest of the paper is organized as follows. Section~\ref{sec:lower}
contains the entropy, temperature oscillation, and primitive
comparison estimates, together with the positive lower bounds for $v$ and
$\theta$. Section~\ref{sec:close} contains the auxiliary elliptic estimate,
the magnetic-pressure oscillation bound, and the uniform upper bound for
$v$, and it closes the volume bootstrap. Section~\ref{sec:high} establishes
the remaining uniform estimates. Section~\ref{sec:full} proves global
continuation and large-time convergence. The appendix records the
regularization argument used to justify the energy calculations.

Let $T_*$ denote the maximal existence time of the solution in
Lemma~\ref{lem:01a}. Unless otherwise stated, the a priori estimates below
are first established on $[0,T]$ with $0<T<T_*$.

\section{Basic estimates and uniform positive lower bounds}\label{sec:lower}
We first derive the entropy bound and its consequences on unit spatial
intervals. Set $\Phi(z)=z-\log z-1$ and
\begin{equation}\label{eq:eta}
\eta=\frac{u^2+|\mathbf{w}|^2+v|\mathbf{b}|^2}{2}+\Phi(v)+\Phi(\theta).
\end{equation}
For the heat-conduction term, define
\begin{equation}\label{eq:K-beta}
\mathcal K_\beta(\xi)=\int_1^\xi s^{\beta-1}(s-1)\dd s.
\end{equation}
The estimates in Remark~\ref{remark:k} imply
\begin{equation}\label{eq:vtb}
 v\theta|\mathbf{b}|^2\leq C\eta+C\mathcal K_\beta(\theta)\eta.
\end{equation}
The corresponding local entropy balance is
\begin{equation}
 \eta_t
 -\left(\frac{(\mathcal K_\beta(\theta))_x}{v}\right)_x
 +\mathcal D_\beta=(\mathcal G_\beta)_x,             \label{eq:ent}
\end{equation}
where
\begin{equation}
 \mathcal D_\beta
 =\frac{\mu u_x^2+|\mathbf{w}_x|^2+|\mathbf{b}_x|^2}{v\theta}
 +\frac{\theta^{\beta-2}\theta_x^2}{v}             \label{eq:D}
\end{equation}
and
\begin{equation}
 \mathcal G_\beta
 =u\left(1-\frac{\theta}{v}-\frac12|\mathbf{b}|^2\right)
 +\frac{\mu uu_x}{v}+\mathbf{w}\cdot \mathbf{b}
 +\frac{\mathbf{w}\cdot \mathbf{w}_x+\mathbf{b}\cdot \mathbf{b}_x}{v}.                \label{eq:G}
\end{equation}

Throughout the a priori estimates, $C>0$ denotes a generic constant
depending only on $\beta$, the initial data, and the chosen boundary
condition. Unless stated otherwise, $C$ is independent of $T$, the unit
interval, the bootstrap bound $M$, and
$\alpha\in[0,\alpha_0]$ under the relevant smallness assumptions.

We shall use the following elementary consequences of the relative
entropy.
\begin{remark}\label{remark:k}
For every $\xi>0$,
\begin{equation}\label{eq:Phi-square}
 |\xi-1|^2\leq C\Phi(\xi)\{1+\Phi(\xi)\}.
\end{equation}
Moreover,
\begin{equation}\label{eq:Phi}
 \xi\leq C\{1+\Phi(\xi)\}.
\end{equation}
In addition, $\mathcal K_\beta(\theta)\geq0$ for $\theta>0$ and
$\mathcal K_\beta(\theta)\geq\mathcal K_0(\theta)=\Phi(\theta)$ for
$\theta\geq1$. Hence
$\theta\leq C\{1+\mathcal K_\beta(\theta)\}$, which, together with
\eqref{eq:eta}, proves \eqref{eq:vtb}.
\end{remark}

\begin{lemma}\label{l:ent}
For every $0\leq t\leq T$,
\begin{equation}
 \int_\Omega\eta(x,t)\dd x
 +\int_0^t\!\int_\Omega \mathcal D_\beta\dd x\dd s
 =E_0,                                                \label{eq:ent-id}
\end{equation}
where
\begin{equation*}
 E_0=\int_\Omega\left\{
 \frac{u_0^2+|\mathbf{w}_0|^2+v_0|\mathbf{b}_0|^2}{2}
 +\Phi(v_0)+\Phi(\theta_0)\right\}\dd x.
\end{equation*}
\end{lemma}
\begin{proof}
Multiply \eqref{eq:mass}, \eqref{eq:mom}, \eqref{eq:trans},
\eqref{eq:ind}, and \eqref{eq:temp} by
\[
 1-\frac1v,\qquad u,\qquad \mathbf{w},\qquad \mathbf{b},\qquad
 1-\frac1\theta,
\]
respectively. Adding the resulting identities and using
\[
 \left(1-\frac1v\right)v_t=\Phi(v)_t
\]
and
\[
 \left(1-\frac1\theta\right)\theta_t=\Phi(\theta)_t
\]
together with
\[
 \mathbf{b}\cdot(v\mathbf{b})_t
 =\frac12\bigl(v|\mathbf{b}|^2\bigr)_t+\frac12|\mathbf{b}|^2u_x
\]
give \eqref{eq:ent}.
For either half-line problem, $\theta_x(0,t)=0$ gives
\[
 \frac{(\mathcal K_\beta(\theta))_x}{v}(0,t)=0.
\]
The flux $\mathcal G_\beta$ vanishes at $x=0$: in
\eqref{eq:bc-N}, this follows from $u=\mathbf{w}=0$ and $\mathbf{b}_x=0$, while in
\eqref{eq:bc-D} it follows from $u=\mathbf{w}=\mathbf{b}=0$. Integrating
\eqref{eq:ent} against spatial cutoffs that converge to one and using
the finite-time strong-solution bounds removes the fluxes at infinity.
Integration in time then yields \eqref{eq:ent-id}.
\end{proof}

For later localization, call an integer $N$ \emph{admissible} if
$N\in\mathbb Z$ when $\Omega=\mathbb R$, and if
$N\in\mathbb N_0=\{0,1,2,\ldots\}$ when $\Omega=\mathbb R_+$. In both
cases, set $I_N=[N,N+1]\subset\overline\Omega$.

\begin{lemma}\label{l:cells}
There exist constants $0<m_0<M_0<\infty$ such that, for every
admissible integer $N$ and every $0\leq t\leq T$,
\begin{subequations}\label{eq:cell}
\begin{equation}\label{eq:cell-v}
 m_0\leq\int_{I_N}v\dd x\leq M_0.
\end{equation}
\begin{equation}\label{eq:cell-theta}
 m_0\leq\int_{I_N}\theta\dd x\leq M_0.
\end{equation}
\end{subequations}
In addition,
\begin{equation}
 \left|\left\{x\in\Omega:\theta(x,t)<\frac12\right\}\right|
 \leq C,\qquad 0\leq t\leq T,                       \label{eq:cold}
\end{equation}
and, for every fixed $K>1$,
\begin{equation}
 \left|\{\theta>K\}\right|
 +\int_{\{\theta>K\}}(\theta+v)\dd x
 \leq C_K,\qquad 0\leq t\leq T.                   \label{eq:hot}
\end{equation}
Here $\{\theta>K\}=\{x\in\Omega:\theta(x,t)>K\}$.
\end{lemma}

\begin{proof}
Jensen's inequality and \eqref{eq:ent-id} give
\[
 \Phi\left(\int_{I_N}v\dd x\right)
 \leq\int_{I_N}\Phi(v)\dd x\leq E_0.
\]
Since $\Phi(z)\to\infty$ as $z\to0^+$ or $z\to\infty$, the 
average of $v$ on $I_N$ lies in a fixed compact subset of $(0,\infty)$. The same
argument applied to $\theta$ proves \eqref{eq:cell-theta}. Moreover,
\[
 E_0\geq\int_{\{\theta<\frac12\}}\Phi(\theta)\dd x
 \geq\Phi\left(\frac12\right)
 \left|\left\{\theta<\frac12\right\}\right|,
\]
which proves \eqref{eq:cold}.

For fixed $K>1$,
\[
 1+\theta\leq C_K\Phi(\theta)
 \quad\text{on }\{\theta>K\}.
\]
Thus \eqref{eq:ent-id} controls both the measure of $\{\theta>K\}$
and the integral of $\theta$ over this set. By \eqref{eq:Phi},
\[
 \int_{\{\theta>K\}}v\dd x
 \leq C\left|\{\theta>K\}\right|
 +C\int_\Omega\Phi(v)\dd x\leq C_K.
\]
This proves \eqref{eq:hot}.
\end{proof}

\begin{lemma}\label{l:osc}
There exist constants $m_\theta,C,c_\theta,C_\theta,\delta>0$ such that,
for almost every $t\in[0,T]$ and every $x\in\Omega$,
\begin{equation}
 \theta(x,t)\geq m_\theta
 -C\int_\Omega\frac{\theta^{\beta-2}\theta_y^2}{v}\dd y.
                                                               \label{eq:osclin}
\end{equation}
In particular, if
\[
 D_\theta(t):=\int_\Omega
 \frac{\theta^{\beta-2}\theta_x^2}{v}\dd x\leq\delta,
\]
then
\begin{equation}
 c_\theta\leq\theta(x,t)\leq C_\theta,
 \qquad x\in\Omega.                                  \label{eq:good-two}
\end{equation}
\end{lemma}

\begin{proof}
Fix an admissible $N$. By \eqref{eq:cell-theta} and continuity on $I_N$,
there exists $y_N(t)\in I_N$ such that
\[
 \theta(y_N(t),t)=\int_{I_N}\theta\dd y\in[m_0,M_0].
\]
If $\beta>0$, then, for $x\in I_N$, the Cauchy--Schwarz inequality and
\eqref{eq:cell-v} yield
\begin{align*}
 \left|\theta^{\frac{\beta}{2}}(x,t)
 -\theta^{\frac{\beta}{2}}(y_N(t),t)\right|
 &\leq\frac\beta2
 \left(\int_{I_N}\frac{\theta^{\beta-2}\theta_y^2}{v}\dd y
 \right)^{\frac12}
 \left(\int_{I_N}v\dd y\right)^{\frac12}\\
 &\leq C_0D_\theta(t)^{\frac12}.
\end{align*}
Choose $\delta_0>0$ so that
$C_0\delta_0^{1/2}\leq\frac12m_0^{\beta/2}$. If
$D_\theta(t)\leq\delta_0$, then $\theta^{\beta/2}(x,t)$ lies in a
fixed compact subinterval of $(0,\infty)$. The map
$z\mapsto z^{2/\beta}$ is Lipschitz on this interval; hence
\[
 |\theta(x,t)-\theta(y_N(t),t)|
 \leq C_1D_\theta(t)^{1/2}.
\]
If $\beta=0$, applying the chain rule to $\log\theta$ gives
\begin{align*}
 |\log\theta(x,t)-\log\theta(y_N(t),t)|
 &\leq
 \left(\int_{I_N}\frac{\theta_y^2}{v\theta^2}\dd y\right)^{1/2}
 \left(\int_{I_N}v\dd y\right)^{1/2}\\
 &\leq C_0D_\theta(t)^{1/2}.
\end{align*}
After decreasing $\delta_0$, the exponential map is Lipschitz on the
resulting compact range, and the same estimate
\[
 |\theta(x,t)-\theta(y_N(t),t)|
 \leq C_1D_\theta(t)^{1/2}
\]
holds when $\beta=0$.

Thus, whenever $D_\theta(t)\leq\delta_0$,
\[
 \theta(x,t)\geq m_0-C_1D_\theta(t)^{1/2}
 \geq\frac{m_0}{2}-\frac{C_1^2}{2m_0}D_\theta(t).
\]
The same oscillation estimate gives
\[
 \theta(x,t)\leq M_0+C_1D_\theta(t)^{1/2}
 \leq M_0+C_1\delta_0^{1/2}.
\]
Set $m_\theta=m_0/2$ and enlarge $C$ so that
$C\geq C_1^2/(2m_0)$ and $C\geq m_\theta/\delta_0$. For
$D_\theta(t)>\delta_0$, the right-hand side of \eqref{eq:osclin} is
nonpositive; hence positivity of $\theta$ proves \eqref{eq:osclin} for
all admissible times. Finally, the bound
$\theta(x,t)\leq M_0+C_1\delta_0^{1/2}$ and \eqref{eq:osclin}
give uniform upper and lower bounds when
$D_\theta(t)\leq\delta_0$. Taking
\[
 \delta\leq\min\left\{\delta_0,\frac{m_\theta}{2C}\right\}
\]
proves \eqref{eq:good-two}.
\end{proof}

We will repeatedly use the following one-dimensional inequalities:
\begin{equation}\label{eq:GN-infty}
 \norm{f}_{L^\infty}^2\leq2\norm{f}_{L^2}\norm{f_x}_{L^2}.
\end{equation}
For $2<p\leq\infty$,
\begin{equation}\label{eq:GN}
 \norm{f}_{L^p}\leq C_p\norm{f}_{L^2}^{\frac12+\frac1p}
 \norm{f_x}_{L^2}^{\frac12-\frac1p},
\end{equation}
valid for $f\in H^1(\Omega)$.

We use a bootstrap assumption only for the specific volume. For a
constant $M\geq e$, suppose temporarily that
\begin{equation}\label{eq:boot}
 M^{-1}\leq v(x,t)\leq M,
 \qquad (x,t)\in\Omega\times[0,T],
\end{equation}
and set $L=\log M$. We shall choose $\alpha$ so that
\begin{equation}\label{eq:alpha-small}
 0\leq\alpha\leq1,\qquad
 \alpha L\leq\delta_*/2,\qquad \alpha L^2\leq1/2,
\end{equation}
where $0<\delta_*\leq1/2$ is fixed later by the elliptic estimate.
The primitive associated with the viscous flux is
\begin{equation}\label{eq:Mv}
 M_\alpha(z)=\int_1^z\frac{\mu(s)}s\dd s
 =\begin{cases}
 (1-z^{-\alpha})/\alpha,&\alpha>0,\\
 \log z,&\alpha=0.
 \end{cases}
\end{equation}
For $y=\log z$, we have
\[
 M_\alpha(e^y)-y=\int_0^y(e^{-\alpha r}-1)\dd r.
\]
Hence the mean value theorem and \eqref{eq:alpha-small} give
\begin{equation}\label{eq:primitive-close}
 |M_\alpha(z)-\log z|
 \leq\tfrac12\alpha e^{\alpha L}L^2\leq1,
 \qquad M^{-1}\leq z\leq M.
\end{equation}
The same assumptions imply
\begin{subequations}\label{eq:mu-basic}
\begin{equation}\label{eq:mu-bounds}
 \tfrac12\leq\mu(v)\leq2.
\end{equation}
Moreover,
\begin{equation}\label{eq:mu-close}
 \norm{\mu(v)-1}_{L^\infty(\Omega\times(0,T))}
 \leq e^{\alpha L}-1\leq\delta_*.
\end{equation}
\end{subequations}
Indeed, since $2\alpha L\leq\delta_*\leq1/2$,
\[
 e^{\alpha L}-1=e^{(2\alpha L)/2}-1\leq2\alpha L\leq\delta_*.
\]
The bound $|\alpha\log v|\leq\alpha L\leq1/4$ also gives
$1/2\leq\mu(v)\leq2$. Exponentiating
\eqref{eq:primitive-close}, we also obtain
\begin{equation}\label{eq:fM}
 e^{-1}z\leq e^{M_\alpha(z)}\leq ez,
 \qquad M^{-1}\leq z\leq M. 
\end{equation}
All constants in \eqref{eq:primitive-close}--\eqref{eq:fM} are
independent of $M$ and $\alpha$.

Set
\begin{equation}\label{eq:Pstar}
 P_* =\frac\theta v+\frac{|\mathbf{b}|^2}{2}.
\end{equation}
The effective viscous flux is
\begin{equation}\label{eq:sigma}
 \sigma=\frac{\mu u_x}{v}-P_*.
\end{equation}
For each admissible interval $I_N=[N,N+1]$, set
% \begin{equation}\label{eq:Bdef}
%  B_N^\alpha(x,t)
%  =\exp\left\{M_\alpha(v_0(x))
%        +\int_N^x\bigl(u(y,t)-u_0(y)\bigr)\,\dd y\right\},
% \end{equation}
% and
% \begin{equation}\label{eq:Ydef}
%  Y_N(t)=\exp\left\{\int_0^t\sigma(N,s)\,\dd s\right\}.
% \end{equation}
For each admissible $N$, define
\begin{equation}\label{eq:Udef}
 U_N(x,t)=\int_N^x u(y,t)\dd y
 -\int_{I_N}\int_N^z u(y,t)\dd y\dd z.
\end{equation}
We also set
\begin{equation}\label{eq:Zdef}
 Z_N(x,t)=\exp\{M_\alpha(v(x,t))-U_N(x,t)\}
\end{equation}
and
\begin{equation}\label{eq:Zbar-def}
 \overline Z_N(t)=\int_{I_N}Z_N\dd x.
\end{equation}

\begin{lemma}\label{l:Kaz}
Under \eqref{eq:boot}--\eqref{eq:alpha-small}, there exist constants
$c,C>0$, independent of $M$, $T$, and $\alpha$, such that, for every
admissible $N$,
\begin{equation}\label{eq:Zbounds}
 c\,v(x,t)\leq Z_N(x,t)\leq C\,v(x,t),
 \qquad (x,t)\in I_N\times[0,T],
\end{equation}
\begin{equation}\label{eq:Zbar-bounds}
 c\leq \overline Z_N(t)\leq C,
 \qquad 0\leq t\leq T,
\end{equation}
and
\begin{equation}\label{eq:Zmean}
 0\leq
 \frac{1}{\overline Z_N(t)}
 \int_{I_N}Z_N(x,t)P_*(x,t)\,\dd x
 \leq C.
\end{equation}
In addition,
\[
 t\longmapsto
 \log\frac{Z_N(\cdot,t)}{\overline Z_N(t)}
\]
belongs to $AC([0,T];C(I_N))$ and satisfies
\begin{equation}\label{eq:Z-evol}
 \partial_t\log\frac{Z_N(x,t)}{\overline Z_N(t)}
 =
 P_*(x,t)
 -
 \frac{1}{\overline Z_N(t)}
 \int_{I_N}Z_N(y,t)P_*(y,t)\,\dd y.
\end{equation}
\end{lemma}

\begin{proof}
By \eqref{eq:ent-id} and the definition of $U_N$ in
\eqref{eq:Udef},
\[
 |U_N(x,t)|
 \leq 2\|u(t)\|_{L^2(\Omega)}
 \leq C.
\]
Since
\[
 Z_N(x,t)=\exp\{M_\alpha(v(x,t))-U_N(x,t)\},
\]
the comparison estimate \eqref{eq:fM} yields \eqref{eq:Zbounds}.
Integrating over $I_N$ and using \eqref{eq:cell-v} gives
\eqref{eq:Zbar-bounds}. Moreover,
\[
 0\leq \int_{I_N}Z_NP_*\,\dd x
 \leq C\int_{I_N}\bigl(\theta+v|\mathbf{b}|^2\bigr)\,\dd x
 \leq C,
\]
and hence \eqref{eq:Zmean} follows from \eqref{eq:Zbar-bounds}.

Since $v_t=u_x$ and $u_t=\sigma_x$, the definition of $U_N$ gives
\[
 (U_N)_t
 =
 \sigma-\int_{I_N}\sigma\,\dd y.
\]
On the other hand,
\[
 (M_\alpha(v))_t
 =
 \frac{\mu(v)v_t}{v}
 =
 \frac{\mu(v)u_x}{v}
 =
 \sigma+P_*.
\]
Therefore,
\[
 (\log Z_N)_t
 =
 P_*+\int_{I_N}\sigma\,\dd y.
\]
Differentiating the definition of $\overline Z_N$ in
\eqref{eq:Zbar-def} gives
\[
 (\log\overline Z_N)_t
 =
 \frac{1}{\overline Z_N}
 \int_{I_N}Z_NP_*\,\dd y
 +\int_{I_N}\sigma\,\dd y.
\]
Subtracting these two identities proves \eqref{eq:Z-evol}.

Finally, on each fixed cell $I_N$, the finite-time regularity in
\eqref{eq:finite} and the bootstrap state bounds justify the above
differentiations and imply
\[
 t\longmapsto
 \log\frac{Z_N(\cdot,t)}{\overline Z_N(t)}
 \in AC([0,T];C(I_N)).
\]
\end{proof}

\begin{proposition}\label{p:vlow}
Under \eqref{eq:boot}--\eqref{eq:alpha-small}, there exists $c>0$,
independent of $M$, $T$, and $\alpha$, such that
\begin{equation}\label{eq:vlow}
 v(x,t)\geq c,\qquad (x,t)\in\Omega\times[0,T].
\end{equation}
\end{proposition}
\begin{proof}
% The representation \eqref{eq:local-rep} is the starting point of the
% lower-bound argument.  Its normalized form \eqref{eq:primitive-rep} has
% the evolution law \eqref{eq:Z-evol}, which is the form used below.

The normalized quantity
\[
R_N(x,t)=\log\frac{Z_N(x,t)}{\overline Z_N(t)}
\]
and its evolution law \eqref{eq:Z-evol} are the starting point of the
lower-bound argument.

Let $D_\theta$ be as in Lemma~\ref{l:osc}, choose $\delta$ from
\eqref{eq:good-two}, and set
\[
 \mathcal B=\{t\in[0,T]:D_\theta(t)>\delta\}.
\]
The entropy estimate gives
\begin{equation}\label{eq:bad-measure}
 |\mathcal B|+\int_{\mathcal B}D_\theta(t)\dd t\leq C.
\end{equation}
Fix an admissible $N$ and $x\in I_N$, and write
$R_N(t)=\log(Z_N(x,t)/\overline Z_N(t))$. By
\eqref{eq:Zbounds} and \eqref{eq:Zbar-bounds},
\[
 v(x,t)\leq Ce^{R_N(t)}.
\]
For almost every $t\notin\mathcal B$, estimate
\eqref{eq:good-two} yields $\theta(x,t)\geq c_\theta$. We may therefore
choose $A>0$, independently of $N$, $x$, and $M$, so that
$R_N(t)<-A$ implies
\[
 \frac{\theta(x,t)}{v(x,t)}
 >\frac1{\overline Z_N(t)}\int_{I_N}Z_NP_*\dd y.
\]
Equation \eqref{eq:Z-evol}, together with \eqref{eq:Zmean} and
$P_*\geq0$, gives, for almost every $t$,
\begin{equation}
 R_N'(t)\geq
 \begin{cases}
 0, & t\notin\mathcal B,\quad R_N(t)<-A,\\
 -C, & t\in\mathcal B\ \text{or}\ R_N(t)\geq-A.
 \end{cases}                                             \label{eq:R-case}
\end{equation}
The initial data, \eqref{eq:Zbounds}, and \eqref{eq:Zbar-bounds} imply
$R_N(0)\geq-C$. Put
$m_N=\min\{-A,R_N(0)\}$. If $(t_-,t_+)$ is a connected component of
$\{t:R_N(t)<m_N\}$, then, for $t\in(t_-,t_+)$,
\begin{align}
 R_N(t)
 &=R_N(t_-)+\int_{t_-}^tR_N'(r)\,\dd r \notag\\
 &\geq m_N-C|\mathcal B\cap(t_-,t)|
 \geq m_N-C|\mathcal B|
 \geq-C.                                             \label{eq:R-lower}
\end{align}
Thus $Z_N\geq c$ and, by \eqref{eq:Zbounds}, $v\geq c$ on $I_N$.
Since $N$ is arbitrary, \eqref{eq:vlow} follows.
\end{proof}

\begin{proposition}\label{p:tlow}
Suppose that $v\geq c_1>0$ and $\mu\geq c_2>0$ on
$\Omega\times[0,T]$, where $c_1$ and $c_2$ are independent of $T$.
For every fixed $\beta\geq0$, there exists a constant $c>0$,
depending on the initial data, $\beta$, $c_1$ and $c_2$, but independent of 
$T$, such that
\begin{equation}
 \theta(x,t)\geq c,
 \qquad (x,t)\in\Omega\times[0,T].                    \label{eq:t-low}
\end{equation}
\end{proposition}

\begin{proof}
Set
\begin{equation}\label{eq:A0}
 A_0=\max\left\{2,\frac{2}{m_\theta},
 \norm{\theta_0^{-1}}_{L^\infty(\Omega)}\right\}.
\end{equation}
Define the inverse-temperature cutoff by
\begin{equation}\label{eq:invcut}
 Z_\theta=(\theta^{-1}-A_0)_+.
\end{equation}
Fix $p>2$. Multiplying \eqref{eq:temp} by $-\theta^{-2}Z_\theta^p$ and
integrating over $\Omega$, we have
\begin{align}
 &\frac1{p+1}\frac{\dd}{\dd t}\int_\Omega Z_\theta^{p+1}\dd x
 +\int_\Omega\frac{\theta^\beta\theta_x^2}{v}
 \left(2\theta^{-3}Z_\theta^p+p\theta^{-4}Z_\theta^{p-1}\right)\dd x
 \notag\\
 &\quad+\int_\Omega
 \frac{\mu u_x^2+|\mathbf{w}_x|^2+|\mathbf{b}_x|^2}{v\theta^2}Z_\theta^p\dd x
 =\int_\Omega\frac{u_x}{v\theta}Z_\theta^p\dd x.          \label{eq:inv-id}
\end{align}
Here the chain rule is applied to the positive part in
\eqref{eq:invcut}; the heat-flux boundary terms vanish by the boundary
conditions. This inverse-temperature test is also used in
\cite{CaoPengSun2021,LuShiXiong2021,TongWangZhang2026}.
Young's inequality gives
\begin{align*}
 \left|\int_\Omega\frac{u_x}{v\theta}Z_\theta^p\dd x\right|
 &\leq\frac12\int_\Omega\frac{\mu u_x^2}{v\theta^2}Z_\theta^p\dd x
 +\frac1{2c_2}\int_\Omega\frac{Z_\theta^p}{v}\dd x\\
 &\leq\frac12\int_\Omega\frac{\mu u_x^2}{v\theta^2}Z_\theta^p\dd x
 +C\int_\Omega Z_\theta^p\dd x.
\end{align*}
After absorbing the first term into \eqref{eq:inv-id} and discarding
the remaining nonnegative terms, we find that
\begin{equation}
 \frac1{p+1}\frac{\dd}{\dd t}\int_\Omega Z_\theta^{p+1}\dd x
 \leq C\int_\Omega Z_\theta^p\dd x.                       \label{eq:inv-Lp}
\end{equation}

Since $A_0\geq2$,
\[
 \{x\in\Omega:Z_\theta(x,t)>0\}
 \subset\{x\in\Omega:\theta(x,t)<\frac12\}.
\]
Thus, \eqref{eq:cold} and H\"older's inequality give
\begin{equation}
 \int_\Omega Z_\theta^p\dd x
 \leq C\left(\int_\Omega Z_\theta^{p+1}\dd x\right)^{\frac p{p+1}},
                                                               \label{eq:invH}
\end{equation}
where $C$ is independent of $p$. Define
\begin{equation}\label{eq:Xp}
 X_p(t)=\int_\Omega Z_\theta^{p+1}(x,t)\dd x.
\end{equation}
Set
\begin{equation}\label{eq:Yp}
 Y_p(t)=X_p(t)^{\frac1{p+1}}
 =\norm{Z_\theta(t)}_{L^{p+1}(\Omega)}.
\end{equation}
The strong-solution regularity and the chain rule for the positive part
imply that $X_p$ is absolutely continuous on every finite time interval.
Hence $Y_p$ is continuous, and
\[
 \mathcal O_p=\{0\leq t<T:Y_p(t)>0\}
\]
is open. On every compact subinterval of $\mathcal O_p$, $X_p$ is
bounded away from zero, so $Y_p$ is absolutely continuous there.
Using \eqref{eq:inv-Lp},
\eqref{eq:invH}, and the chain rule, we obtain
\begin{equation}
 Y_p'(t)
 =\frac1{p+1}X_p(t)^{-\frac p{p+1}}X_p'(t)
 \leq C                                                \label{eq:Ypdt}
\end{equation}
for almost every $t\in\mathcal O_p$, where $C$ is independent of $p$
and $T$.

With $\delta$ as in Lemma~\ref{l:osc}, set
\[
 \mathcal B  =\left\{0\leq t<T:
 \int_\Omega\frac{\theta^{\beta-2}\theta_x^2}{v}\dd x>\delta
 \right\}.
\]
By \eqref{eq:ent-id} and Chebyshev's inequality, we obtain
\[
 |\mathcal B|\leq\frac{E_0}{\delta}.
\]
The choice of $\delta$ and \eqref{eq:osclin} also imply that, for
almost every $t\notin\mathcal B$,
\begin{equation}
 \inf_{x\in\Omega}\theta(x,t)\geq\frac{m_\theta}{2},
                                                        \label{eq:good}
\end{equation}
 and hence $Z_\theta(\cdot,t)=0$ by \eqref{eq:invcut}. Consequently,
\[
 |\mathcal O_p\setminus\mathcal B|=0.
\]

Fix $t\in[0,T)$. If $Y_p(t)=0$, there is nothing to prove.
Otherwise, let $(t_-,t_+)$ be the connected component of $\mathcal O_p$
containing $t$. The choice
$A_0\geq\norm{\theta_0^{-1}}_{L^\infty}$ in
\eqref{eq:invcut} gives $Y_p(0)=0$.
By continuity, $Y_p(t_-)=0$.
For $0<\varepsilon<t-t_-$, integrating \eqref{eq:Ypdt}
over $[t_-+\varepsilon,t]$ yields
\[
 Y_p(t)\leq Y_p(t_-+\varepsilon)+C(t-t_- -\varepsilon).
\]
Letting $\varepsilon\to0^+$, we obtain
\[
 Y_p(t)\leq C(t-t_-).
\]
Since $(t_-,t)\subset\mathcal O_p$ and
$|\mathcal O_p\setminus\mathcal B  |=0$,
\[
 t-t_-=|\mathcal B  \cap(t_-,t)|\leq|\mathcal B  |.
\]
It follows that
\[
 Y_p(t)\leq\frac{CE_0}{\delta},
\]
with a constant independent of $p$ and $t$.
Since the bound is independent of $p$, letting $p\to\infty$ gives
\[
 \norm{Z_\theta(t)}_{L^\infty(\Omega)}\leq C,
 \qquad 0\leq t<T.
\]
Finally, \eqref{eq:invcut} gives
\[
 \theta^{-1}(x,t)\leq A_0+C,
 \qquad (x,t)\in\Omega\times[0,T).
\]
The continuity of $\theta$ at $t=T$ completes the proof.
\end{proof}

\section{A uniform upper bound for the specific volume}\label{sec:close}

The upper bound for $v$ requires a time-integrated estimate of the
magnetic-pressure oscillation in \eqref{eq:Z-evol}.  We obtain it from an
auxiliary elliptic problem.  Set $\mathcal V=H^1(\Omega)$ and
$\mathcal V^*=(H^1(\Omega))^*$.
\begin{lemma}\label{l:res}
Fix $0<T<T_*$. For almost every $t\in(0,T)$, there exists a unique
$\psi(t)\in H^1(\Omega)$ satisfying
\begin{equation}
 \int_\Omega\left(\frac{\psi_x\varphi_x}{v}+\psi\varphi\right)\dd x
 =\int_\Omega\eta\varphi\dd x
 \qquad\text{for every }\varphi\in H^1(\Omega).          \label{eq:reswk}
\end{equation}
Define
\begin{equation}\label{eq:Fdef}
 F=-\frac{\psi_x}{v}.
\end{equation}
Set
\begin{equation}\label{eq:FH}
 \mathcal H(t)=\frac12\int_\Omega
 \left(\frac{\psi_x^2}{v}+\psi^2\right)\dd x.
\end{equation}
Then there is a constant $C>0$, independent of $T$, the finite-time
bounds for $v$, $M$, and $\alpha$, such that
\begin{equation}\label{eq:psi-positive}
 \psi\geq0.
\end{equation}
The solution has the same mass as $\eta$:
\begin{equation}\label{eq:res-L1}
 \int_\Omega\psi\dd x=\int_\Omega\eta\dd x.
\end{equation}
Moreover,
\begin{equation}\label{eq:Fx}
 F_x=\eta-\psi,
\end{equation}
\begin{equation}\label{eq:F-inf}
 \norm F_{L^\infty(\Omega)}\leq E_0,
\end{equation}
\begin{equation}\label{eq:resinf}
 \norm\psi_{L^\infty(\Omega)}\leq C,
\end{equation}
and
\begin{equation}\label{eq:resst}
 \mathcal H(t)+\int_\Omega F^2\dd x\leq C.
\end{equation}
On the whole line, $F$ has a continuous representative that tends to zero
at both infinities.  On the half-line, $F(0,t)=0$ and
$F(x,t)\to0$ as $x\to\infty$.  Moreover,
$\mathcal H\in W^{1,1}(0,T)$, and the following identity holds in
$\mathcal D'(0,T)$:
\begin{equation}
 \mathcal H_t
 = \int_\Omega F\left(\mathcal K_\beta(\theta)\right)_x\dd x+\int_\Omega vF\mathcal G_\beta\dd x
 -\int_\Omega\mathcal D_\beta\psi\dd x
 +\frac12\int_\Omega u_xF^2\dd x.                 \label{eq:rest}
\end{equation}
\end{lemma}

\begin{proof}
On the fixed interval $[0,T]$, Lemma~\ref{lem:01a} gives
$0<c_T\leq v\leq C_T$.  Hence the bilinear form
\[
 a(\phi,\varphi;t)
 =\int_\Omega\left(\frac{\phi_x\varphi_x}{v}
 +\phi\varphi\right)\dd x
\]
is continuous on $H^1(\Omega)\times H^1(\Omega)$ and satisfies
\[
 a(\phi,\phi;t)\geq
 \min\{1,C_T^{-1}\}\norm{\phi}_{\mathcal V}^2.
\]
Since $\eta\in L^1(\Omega)$ and
$H^1(\Omega)\hookrightarrow L^\infty(\Omega)$ in one dimension, the
right-hand side of \eqref{eq:reswk} defines an element of $\mathcal V^*$.
The Lax--Milgram theorem therefore gives a unique $\psi\in H^1(\Omega)$.
Testing \eqref{eq:reswk} with $\min\{\psi,0\}$ and using $\eta\geq0$
shows that $\psi\geq0$.

Choose $0\leq\chi_R\leq1$ such that $\chi_R\uparrow1$ and
$\norm{\chi_R'}_{L^2}\to0$ as $R\to\infty$.  Testing
\eqref{eq:reswk} with $\chi_R$, the gradient term tends to zero by
Cauchy--Schwarz and $c_T\leq v\leq C_T$.  Monotone convergence for
$\psi\chi_R$ and dominated convergence for $\eta\chi_R$ then give
\begin{equation}
 \int_\Omega\psi\dd x=\int_\Omega\eta\dd x\leq E_0. \label{eq:psi-L1}
\end{equation}
The weak equation gives \eqref{eq:Fx} in distributions. On the
half-line, its natural boundary condition is $F(0,t)=0$.  On the whole
line, $F\in L^2(\mathbb R)$ and $F_x\in L^1(\mathbb R)$; hence its
limits at both infinities exist and, by the $L^2$ condition, are zero.
Consequently,
\begin{equation}\label{eq:F-rep}
  F(x,t)=
 \begin{cases}
 \displaystyle\int_{-\infty}^x(\eta-\psi)\dd y,
     &\Omega=\mathbb R,\\[2mm]
 \displaystyle\int_0^x(\eta-\psi)\dd y,
     &\Omega=\mathbb R_+.
  \end{cases}
\end{equation}
Since both $\eta$ and $\psi$ are nonnegative and have the same mass,
\eqref{eq:F-rep} and \eqref{eq:psi-L1} imply
$\norm F_{L^\infty}\leq E_0$.  In the half-line case, equality of the
two masses also gives $F(x,t)\to0$ as $x\to\infty$.

For each admissible integer $N$,
$\int_N^{N+1}\psi\dd x\leq E_0$, so there exists
$x_N\in[N,N+1]$ such that $\psi(x_N,t)\leq E_0$.
Since $\psi_x=-vF$, for $x\in[N,N+1]$ we have
\[
 |\psi(x,t)-\psi(x_N,t)|
 \leq\norm F_{L^\infty}\int_N^{N+1}v\dd y
 \leq E_0M_0.
\]
It follows that $\norm\psi_{L^\infty}\leq E_0(1+M_0)$.  Testing
\eqref{eq:reswk} with $\psi$ now yields
\[
 2\mathcal H(t)=\int_\Omega\eta\psi\dd x
 \leq E_0\norm\psi_{L^\infty}\leq C.
\]
To remove the weight $v$ from the estimate of $F$, split $\Omega$ into
$\{v\geq\frac12\}$ and $\{v<\frac12\}$.  Since
$\Phi(v)\geq\Phi(\frac12)$ on the latter set,
$|\{v<\frac12\}|\leq E_0/\Phi(\frac12)$.  Therefore,
\[
 \int_\Omega F^2\dd x \leq2\int_\Omega vF^2\dd x +\norm F_{L^\infty}^2|\{v<\frac12\}|
 \leq4\mathcal H(t)+\frac{E_0^3}{\Phi(\frac12)}\leq C.
\]
This proves \eqref{eq:F-inf}, \eqref{eq:resinf}, and \eqref{eq:resst};
in particular, their
constants do not depend on the finite-time bounds $c_T$ and $C_T$ used
to construct $\psi$.

We next justify the time regularity of $\mathcal H$ and the identity
\eqref{eq:rest}.  By \eqref{eq:finite} and the one-dimensional Sobolev
embedding,
\[
 (v^{-1})_t=-\frac{v_t}{v^2}\in L^2(0,T;L^\infty(\Omega)).
\]
The finite-time strong-solution bounds also imply
\[
\mathcal G_\beta\in L^2(0,T;L^2(\Omega)).
\]
They also give
\[
\frac{(\mathcal K_\beta(\theta))_x}{v}
=\frac{\theta^{\beta-1}(\theta-1)\theta_x}{v}
\in L^2(0,T;L^2(\Omega)).
\]
while \eqref{eq:ent-id} yields
\[
\mathcal D_\beta\in L^1(\Omega\times(0,T)).
\]
Pairing \eqref{eq:ent} with $\varphi\in H^1(\Omega)$ and using the
boundary conditions, we have
\[
\begin{aligned}
\langle\eta_t,\varphi\rangle
={}&-\int_\Omega
\left\{\mathcal G_\beta+
\frac{(\mathcal K_\beta(\theta))_x}{v}\right\}
\varphi_x\,\dd x
-\int_\Omega\mathcal D_\beta\varphi\,\dd x .
\end{aligned}
\]
The embedding $H^1(\Omega)\hookrightarrow L^\infty(\Omega)$ therefore
gives
\[
\|\eta_t\|_{\mathcal V^*}
\leq
\|\mathcal G_\beta\|_{L^2}
+\left\|\frac{(\mathcal K_\beta(\theta))_x}{v}\right\|_{L^2}
+C\|\mathcal D_\beta\|_{L^1},
\]
and hence
\[
 v^{-1}\in W^{1,2}(0,T; L^\infty(\Omega)).
\]
Moreover,
\[
 \eta\in W^{1,1}(0,T; \mathcal V^*).
\]
Let $E_T\subset(0,T)$ be a set of full measure on which
\eqref{eq:reswk}, \eqref{eq:resst}, and \eqref{eq:ent} hold.  For
$s,t\in E_T$ with $s<t$, subtract \eqref{eq:reswk} at these two times gives
\[
 \int_\Omega\left\{ \frac{(\psi(t)-\psi(s))_x\varphi_x}{v(t)}
 +(\psi(t)-\psi(s))\varphi\right\}\dd x =\langle\eta(t)-\eta(s),\varphi\rangle
 -\int_\Omega\bigl(v^{-1}(t)-v^{-1}(s)\bigr) \psi_x(s)\varphi_x\dd x.
\]
Taking $\varphi=\psi(t)-\psi(s)$ and using the coercivity of
$a(\cdot,\cdot;t)$ on the fixed interval, we find
\begin{align*}
 C_T^{-1}\norm{\psi(t)-\psi(s)}_{\mathcal V}^2
 &\leq\left\{\norm{\eta(t)-\eta(s)}_{\mathcal V^*}
 +\norm{v^{-1}(t)-v^{-1}(s)}_{L^\infty}
 \norm{\psi_x(s)}_{L^2}\right\}
 \norm{\psi(t)-\psi(s)}_{\mathcal V}.
\end{align*}
The bound for $\mathcal H$ and $v\leq C_T$ imply
$\sup_{r\in E_T}\norm{\psi_x(r)}_{L^2}\leq C_T$.  Hence
\begin{align*}
 \norm{\psi(t)-\psi(s)}_{\mathcal V}
 &\leq C_T\left\{
 \norm{\eta(t)-\eta(s)}_{\mathcal V^*}
 +\norm{v^{-1}(t)-v^{-1}(s)}_{L^\infty}\right\}
 \\
 &\leq C_T\int_s^t\left\{
 \norm{\partial_r\eta(r)}_{\mathcal V^*}
 +\norm{\partial_r(v^{-1})(r)}_{L^\infty}\right\}\dd r.
\end{align*}
Thus $t\mapsto\psi(t)$, initially defined on $E_T$, has a unique
extension in $W^{1,1}(0,T;H^1(\Omega))$, still denoted by $\psi$.
The time continuity of $v^{-1}$, $\eta$, and $\psi$ extends
\eqref{eq:reswk} and the uniform bound for $\mathcal H$ to every
$t\in[0,T]$.

The regularities just established permit differentiation of
\eqref{eq:reswk} in $\mathcal V^*$.  For almost every $t$ and every
$\varphi\in H^1(\Omega)$, this gives
\[
 \int_\Omega\left(\frac{\psi_{xt}\varphi_x}{v}
 +\psi_t\varphi\right)\dd x
 =\langle\eta_t,\varphi\rangle
 +\int_\Omega\frac{u_x\psi_x\varphi_x}{v^2}\dd x.
\]
Taking $\varphi=\psi(t)$ and applying the product rule to
\eqref{eq:FH}, with $(v^{-1})_t=-u_x/v^2$, yields
\begin{equation}\label{eq:H-derivative}
 \frac{\dd}{\dd t}\mathcal H =\int_\Omega\left(
 \frac{\psi_x\psi_{xt}}{v}+\psi\psi_t\right)\dd x
 -\frac12\int_\Omega\frac{u_x\psi_x^2}{v^2}\dd x =\langle\eta_t,\psi\rangle
 +\frac12\int_\Omega u_xF^2\dd x.
\end{equation}
It remains to identify $\langle\eta_t,\psi\rangle$.  Pair
\eqref{eq:ent} with $\psi$.  On the half-line, the boundary terms at
$x=0$ vanish because
$\theta_x(0,t)=0$ and $\mathcal G_\beta(0,t)=0$.
The terms at infinity are handled with spatial cutoffs.  Since
$\psi_x=-vF$,
\begin{align*}
 \langle\eta_t,\psi\rangle &
 =-\int_\Omega\frac{\left(\mathcal K_\beta(\theta)\right)_x}{v}\psi_x\dd x-\int_\Omega\mathcal G_\beta\psi_x\dd x
 -\int_\Omega\mathcal D_\beta\psi\dd x
 \\&=\int_\Omega F\left(\mathcal K_\beta(\theta)\right)_x\dd x+\int_\Omega vF\mathcal G_\beta\dd x
 -\int_\Omega\mathcal D_\beta\psi\dd x.
\end{align*}
Substitution into \eqref{eq:H-derivative} proves
\eqref{eq:rest}.  Although $\psi$ was constructed on $[0,T]$, the
bounds \eqref{eq:F-inf}, \eqref{eq:resinf}, and \eqref{eq:resst} are
independent of $T$.
\end{proof}

We now extract from \eqref{eq:rest} the coercive quantities needed for
the magnetic-pressure estimate.  Substituting \eqref{eq:G} into
\eqref{eq:rest} gives
\begin{align}
 &\mathcal H_t-\int_\Omega F\bigl(\mathcal K_\beta(\theta)\bigr)_x\dd x+\int_\Omega\mathcal D_\beta\psi\dd x
 \notag\\&=\int_\Omega Fu\left(v-\theta-\frac12v|\mathbf{b}|^2\right)\dd x
 +\int_\Omega\mu Fuu_x\dd x
 +\int_\Omega Fv\,\mathbf{w}\cdot \mathbf{b}\dd x\notag\\
 &\quad+\int_\Omega F\mathbf{w}\cdot \mathbf{w}_x\dd x
 +\int_\Omega F\mathbf{b}\cdot \mathbf{b}_x\dd x
  +\frac12\int_\Omega u_xF^2\dd x. \label{eq:res-flux}
\end{align}
By \eqref{eq:Fx}, integration by parts yields
\begin{equation}\label{eq:res-parts-k}
 \int_\Omega F\bigl(\mathcal K_\beta(\theta)\bigr)_x\dd x
 =-\int_\Omega(\eta-\psi)\mathcal K_\beta(\theta)\dd x.
\end{equation}
For the term containing $\mathbf{w}$, \eqref{eq:Fx} and integration by parts give
\begin{equation}\label{eq:res-parts-w}
 \int_\Omega F\mathbf{w}\cdot \mathbf{w}_x\dd x
 =-\frac12\int_\Omega(\eta-\psi)|\mathbf{w}|^2\dd x.
\end{equation}
Finally,
\begin{equation}\label{eq:res-parts-b}
 \int_\Omega F\mathbf{b}\cdot \mathbf{b}_x\dd x
 =-\frac12\int_\Omega(\eta-\psi)|\mathbf{b}|^2\dd x.
\end{equation}
For the half-line problems, the boundary terms at $x=0$ vanish because
$F(0,t)=0$; the terms at infinity are justified by spatial cutoffs.
The elementary bounds \eqref{eq:Phi-square} and \eqref{eq:Phi} also give
\begin{align} \left|v-\theta-\frac12v|\mathbf{b}|^2-\eta\right|^2&=\left| (v-1)-(\theta-1)-\frac12v|\mathbf{b}|^2-\eta \right|^2
 \leq C\eta(1+\eta).        \label{eq:resrem}
\end{align}
We shall also need to absorb the term containing
$\mathcal K_\beta(\theta)\psi$.  Given $\varepsilon>0$, choose
$\Lambda>0$, independently of time, so that
$\norm\psi_{L^\infty}/\Lambda\leq\varepsilon$; this is possible by
\eqref{eq:resinf}.  On $\{\eta<\Lambda\}$, the inequality
$\Phi(\theta)\leq\eta<\Lambda$ confines $\theta$ to a compact subset
of $(0,\infty)$.  Thus
$\mathcal K_\beta(\theta)\leq C_{\Lambda,\beta}$ on this set, and
\eqref{eq:res-L1} yields
\begin{equation}
\begin{aligned}
 \int_\Omega\mathcal K_\beta(\theta)\psi\dd x
 &\leq\frac{\norm\psi_{L^\infty}}{\Lambda}
 \int_{\{\eta\geq\Lambda\}}
 \mathcal K_\beta(\theta)\eta\dd x
 +C_{\Lambda,\beta}\int_{\{\eta<\Lambda\}}\psi\dd x\\
 &\leq\varepsilon\int_\Omega
 \mathcal K_\beta(\theta)\eta\dd x+C_\varepsilon.
\end{aligned}
\label{eq:Kpsi}
\end{equation}

We shall use the following estimate. For $0\leq s<t\leq T$,
\begin{equation}
 \int_s^t\!\int_\Omega\frac{\mu(v)u_x^2}{v}\dd x\dd r
 \leq C+C\int_s^t\!\int_\Omega\mathcal K_\beta(\theta)\eta\dd x\dd r.
 \label{eq:heat-conv}
\end{equation}
Estimate \eqref{eq:heat-conv} converts the longitudinal viscous dissipation into the
quantity on the right-hand side of \eqref{eq:Kpsi}.  Set
$\chi(\xi)=\left(1-\frac{2}{\xi}\right)_+$ and
\[
 H(\xi)=
 \begin{cases}
 0,&\xi\leq2,\\
 \xi-2-2\log\frac{\xi}{2},&\xi>2.
 \end{cases}
\]
Then $H'=\chi$ and $0\leq H\leq C\Phi$.  Testing
\eqref{eq:temp} with $\chi(\theta)$ and integrating the heat-conduction
term by parts give
\begin{align*}
 \int_\Omega\frac{\mu(v)u_x^2+|\mathbf{w}_x|^2+|\mathbf{b}_x|^2}{v}\chi(\theta)\dd x
 &=\frac{\dd}{\dd t}\int_\Omega H(\theta)\dd x
 +2\int_{\{\theta>2\}}\frac{\theta^{\beta-2}\theta_x^2}{v}\dd x\\
 &\quad+\int_\Omega\frac{\theta u_x}{v}\chi(\theta)\dd x.
\end{align*}
On the support of $\chi$, Young's inequality and \eqref{eq:vlow} give
\[
 \left|\int_\Omega\frac{\theta u_x}{v}\chi(\theta)\dd x\right|
 \leq\frac12\int_\Omega\frac{\mu(v)u_x^2}{v}\chi(\theta)\dd x
 +C\int_\Omega\mathcal K_\beta(\theta)\eta\dd x.
\]
Here we used $\mu(v)\geq1/2$ and
$\theta^2\leq C\Phi(\theta)^2
\leq C\mathcal K_\beta(\theta)\eta$ on $\{\theta>2\}$.
Integrating in time and using \eqref{eq:ent-id} controls the left-hand
side of \eqref{eq:heat-conv} on $\{\theta>4\}$, where $\chi\geq1/2$.
On $\{\theta\leq4\}$,
\begin{equation}\label{eq:heat-low}
 \frac{\mu u_x^2}{v}
 \leq4\frac{\mu u_x^2}{v\theta},
\end{equation}
The right-hand side of \eqref{eq:heat-low} is controlled by the entropy
dissipation in \eqref{eq:ent-id}.
Combining the two regions proves \eqref{eq:heat-conv}.

Write $\mu=1+(\mu-1)$.  The constant part produces the cancellation
used in the constant-viscosity case, whereas the second part contributes
the error term $\int_\Omega(\mu-1)Fuu_x\dd x$.
\begin{proposition}\label{p:bcap}
If $\delta_*$ in \eqref{eq:alpha-small} is sufficiently small,
depending only on the data and $\beta$, then for $0\leq s<t\leq T$,
\begin{equation}\label{eq:Keta-int}
\int_s^t\!\int_\Omega
\left\{\mathcal K_\beta(\theta)\eta+\eta u^2\right\}\dd x\dd r
\leq C(1+t-s),
\end{equation}
and
\begin{equation}\label{eq:b-cap}
\int_s^t\!\int_\Omega v\theta|\mathbf{b}|^2\dd x\dd r\leq C(1+t-s).
\end{equation}
\end{proposition}
\begin{proof}
Throughout this proof, unmarked spatial integrals are over $\Omega$.
Integration by parts, using \eqref{eq:Fx}, gives the exact
constant-viscosity cancellation
\begin{equation}\label{eq:constant-cancel}
\begin{aligned}
\int Fuu_x+\tfrac12\int u_xF^2
&=-\tfrac12\int\eta u^2-\int uF\eta
+\tfrac12\int\psi u^2+\int uF\psi.
\end{aligned}
\end{equation}
The boundary terms vanish by the boundary behavior of $F$ established
in Lemma~\ref{l:res}, together with the finite-energy bounds.  Combining
\eqref{eq:constant-cancel} with \eqref{eq:res-flux} and the identities
\eqref{eq:res-parts-k}, \eqref{eq:res-parts-w}, and
\eqref{eq:res-parts-b}, we obtain
\begin{align}\label{eq:res-id}
&\mathcal H_t+\int_\Omega\mathcal K_\beta(\theta)\eta\dd x
+\tfrac12\int_\Omega\eta(u^2+|\mathbf{w}|^2+|\mathbf{b}|^2)\dd x
+\int_\Omega\mathcal D_\beta\psi\dd x\notag\\
&\quad=\int_\Omega Fu\left(v-\theta-\tfrac12v|\mathbf{b}|^2-\eta\right)\dd x
+\int_\Omega Fv \mathbf{w}\cdot \mathbf{b}\dd x\notag\\
&\qquad+\int_\Omega\mathcal K_\beta(\theta)\psi\dd x
+\int_\Omega uF\psi\dd x\notag\\
&\qquad+\tfrac12\int_\Omega\psi(u^2+|\mathbf{w}|^2+|\mathbf{b}|^2)\dd x
+\int_\Omega(\mu-1)Fuu_x\dd x.
\end{align}
For the first three terms on the right-hand side of \eqref{eq:res-id},
\begin{align*}
\left|\int Fu\left(v-\theta-\frac12v|\mathbf{b}|^2-\eta\right)\dd x\right|
 &\leq \varepsilon\int_\Omega\eta u^2\dd x
 +C_\varepsilon\int_\Omega F^2(1+\eta)\dd x
 \leq \varepsilon\int_\Omega\eta u^2\dd x+C_\varepsilon,\\
\left|\int Fv\,\mathbf{w}\cdot \mathbf{b}\dd x\right|
&\leq \varepsilon\int_\Omega\eta|\mathbf{w}|^2\dd x
 +C_\varepsilon\int_\Omega vF^2\dd x
 \leq \varepsilon\int_\Omega\eta|\mathbf{w}|^2\dd x+C_\varepsilon,\\
\left|\int\mathcal K_\beta(\theta)\psi\dd x\right|
 &\leq\varepsilon\int_\Omega\mathcal K_\beta(\theta)\eta\dd x+C_\varepsilon.
\end{align*}
Here we used \eqref{eq:resrem}, $\eta\geq v|\mathbf{b}|^2/2$,
 \eqref{eq:F-inf}, \eqref{eq:resinf}, \eqref{eq:resst}, and
 \eqref{eq:Kpsi}. The two
remaining terms satisfy
\begin{equation*}
\left|\int uF\psi\dd x\right|
 +\frac12\left|\int\psi(u^2+|\mathbf{w}|^2+|\mathbf{b}|^2)\dd x\right|
\leq C
\end{equation*}
by \eqref{eq:F-inf}, \eqref{eq:resinf}, \eqref{eq:ent-id}, and
\eqref{eq:vlow}.
Choosing $\varepsilon$ so that the
$\int_\Omega\mathcal K_\beta(\theta)\eta\dd x$,
$\int_\Omega\eta u^2\dd x$, and
$\int_\Omega\eta|\mathbf{w}|^2\dd x$ contributions are absorbed into the left-hand side of
\eqref{eq:res-id}, and dropping the remaining nonnegative terms, we obtain
\begin{equation}\label{eq:HK-eta}
\mathcal H_t+c\int_\Omega\mathcal K_\beta(\theta)\eta\dd x
+c\int_\Omega\eta u^2\dd x
\leq C+\left|\int_\Omega(\mu-1)Fuu_x\dd x\right|.
\end{equation}
It remains to control $\int_\Omega(\mu-1)Fuu_x\dd x$ without using
$v_x$.  For
every $\varepsilon>0$, Young's inequality and \eqref{eq:mu-basic} give
\begin{align}\label{eq:new-error}
\left|\int_\Omega(\mu-1)Fuu_x\dd x\right|
&\leq\varepsilon\int_\Omega\frac{\mu u_x^2}{v}\dd x
+C\varepsilon^{-1}\delta_*^2
 \int_\Omega\frac v\mu F^2u^2\dd x\notag\\
&\leq\varepsilon\int_\Omega\frac{\mu u_x^2}{v}\dd x
+C\varepsilon^{-1}\delta_*^2
+C\varepsilon^{-1}\delta_*^2\int_\Omega\eta u^2\dd x.
\end{align}
Indeed, $v\leq C(1+\Phi(v))$, $\mu\geq1/2$, and
$\norm F_{L^\infty}\leq C$ imply
$\int_\Omega(v/\mu)F^2u^2\dd x
\leq C\left(1+\int_\Omega\eta u^2\dd x\right)$.  Thus no derivative of $\mu$, and in
particular no estimate of $v_x$, enters \eqref{eq:new-error}.
The estimate \eqref{eq:heat-conv} gives
\begin{equation*}
\int_s^t\!\int_\Omega\frac{\mu u_x^2}{v}\dd x\dd r
\leq C_h\left(1+
\int_s^t\!\int_\Omega\mathcal K_\beta(\theta)\eta\dd x\dd r\right).
\end{equation*}
Integrating \eqref{eq:HK-eta} over $(s,t)$, using
$0\leq\mathcal H\leq C$, and applying \eqref{eq:new-error} and
the estimate \eqref{eq:heat-conv}, we find
\begin{align*}
 c\int_s^t\!\int_\Omega\mathcal K_\beta(\theta)\eta\dd x\dd r
 +c\int_s^t\!\int_\Omega\eta u^2\dd x\dd r
&\leq C(1+t-s)\\
&\quad+\varepsilon C_h\int_s^t\!\int_\Omega\mathcal K_\beta(\theta)\eta\dd x\dd r
+C\varepsilon^{-1}\delta_*^2
\int_s^t\!\int_\Omega\eta u^2\dd x\dd r.
\end{align*}
Choose first $\varepsilon>0$ so that $\varepsilon C_h\leq c/2$ and
then $0<\delta_*\leq1/2$ so that
$C\varepsilon^{-1}\delta_*^2\leq c/2$. Both choices depend only on
the data and $\beta$, and are independent of $M$ and $\alpha$.
Absorbing the two time integrals on the right-hand side proves
\eqref{eq:Keta-int}.  Finally,
integrating \eqref{eq:vtb} over $\Omega\times(s,t)$ and using
$\int_\Omega\eta\dd x\leq E_0$ gives \eqref{eq:b-cap}.
\end{proof}

\begin{lemma}\label{l:bP}
For all $0\leq s<t\leq T$, every admissible integer $N$, and every
$x\in[N,N+1]$,
\begin{equation}
 \int_s^t\left|
 \frac12|\mathbf{b}(x,r)|^2-\frac12\int_N^{N+1}|\mathbf{b}(y,r)|^2\dd y
 \right|\dd r
 \leq C\sqrt{1+t-s}.                                \label{eq:bP}
\end{equation}
In fact,
\begin{equation}
 \int_s^t\!\int_N^{N+1}|\mathbf{b}||\mathbf{b}_y|\dd y\dd r
 \leq C\sqrt{1+t-s}. \label{eq:b-osc}
\end{equation}
\end{lemma}

\begin{proof}
Fix $r$.  The fundamental theorem of calculus on the unit interval
$[N,N+1]$ gives, for $x\in[N,N+1]$,
\begin{align*}
 &\left|\frac12|\mathbf{b}(x,r)|^2
 -\frac12\int_N^{N+1}|\mathbf{b}(y,r)|^2\dd y\right|
 \\ &\quad\leq\int_N^{N+1}|\mathbf{b}||\mathbf{b}_y|\dd y\\
 &\quad\leq\left(\int_N^{N+1}v\theta|\mathbf{b}|^2\dd y\right)^{\frac12}
 \left(\int_N^{N+1}\frac{|\mathbf{b}_y|^2}{v\theta}\dd y\right)^{\frac12}.
\end{align*}
Applying Cauchy--Schwarz also in time and then using \eqref{eq:b-cap}
and \eqref{eq:ent-id}, we obtain
\begin{align*}
& \int_s^t\!\int_N^{N+1}|\mathbf{b}||\mathbf{b}_y|\dd y\dd r
 \\ &\leq\left(\int_s^t\!\int_\Omega v\theta|\mathbf{b}|^2\dd x\dd r\right)^{\frac12}
 \left(\int_s^t\!\int_\Omega
 \frac{|\mathbf{b}_x|^2}{v\theta}\dd x\dd r\right)^{\frac12}\\
 &\leq C\sqrt{1+t-s}.
\end{align*}
This proves \eqref{eq:b-osc}, and its pointwise precursor gives
\eqref{eq:bP}.
\end{proof}

\begin{proposition}\label{p:v-up}
There exists a constant $C$, independent of $T$, such that
\begin{equation}
 v(x,t)\leq C,\qquad x\in\Omega,\quad 0\leq t\leq T.
 \label{eq:v-up}
\end{equation}
\end{proposition}

\begin{proof}
 Let
 \[
 \mathcal B=\{t\in[0,T]:D_\theta(t)>\delta\}.
 \]
 By \eqref{eq:ent-id},
 \begin{equation}\label{eq:bad-measure-up}
 |\mathcal B|+\int_{\mathcal B}D_\theta(t)\dd t\leq C.
 \end{equation}
 We show that the logarithmic  oscillation of $Z_N$ on $I_N$ is uniformly bounded.
For an admissible $N$, let
\[
 L_N(t)=\max_{x\in [N,N+1]}\log Z_N(x,t)-\log\overline Z_N(t),
\]
and choose $x_+(t)\in[N,N+1]$ at which the maximum is attained.  We
estimate the magnetic-pressure contribution through the normalized
identity \eqref{eq:Z-evol}.
By \eqref{eq:vlow}, \eqref{eq:fM}, and \eqref{eq:Zbounds},
\begin{equation}\label{eq:Zmax-v}
 \max_{[N,N+1]}v\leq C v(x_+,t).
\end{equation}
Moreover, \eqref{eq:vlow} gives
\begin{equation}\label{eq:v-at-max}
 v(x_+,t)\geq c.
\end{equation}
By continuity and \eqref{eq:cell-theta}, there is $x_0\in[N,N+1]$ such that
$\theta(x_0,t)$ equals its average on $I_N$.  Applying the fundamental
theorem of calculus to $\sqrt\theta$ between $x_0$ and $x_+$ gives
\begin{align*}
 \sqrt{\theta(x_+,t)}
 &\leq C+\frac12\int_N^{N+1}\frac{|\theta_x|}{\sqrt\theta}\dd x\\
 &\leq C+C\left(\int_N^{N+1}\frac{\theta_x^2}{v\theta^2}\dd x\right)^{\frac12}
             \left(\int_N^{N+1}v\theta\dd x\right)^{\frac12}\\
 &\leq C+C\left\{D_\theta(t)\max_{[N,N+1]}v\right\}^{\frac12}.
\end{align*}
Here \eqref{eq:cell-theta} controls $\int_{N}^{N+1}v\theta$ by
$C\max_{[N,N+1]}v$, and \eqref{eq:t-low} allows us to replace
$\theta_x^2/(v\theta^2)$ by a constant multiple of the integrand in
$D_\theta$.  Hence \eqref{eq:Zmax-v} implies
\begin{equation}
 \frac{\theta(x_+,t)}{v(x_+,t)}\leq C(1+D_\theta(t)).
 \label{eq:max-pressure}
\end{equation}

For almost every $t\notin\mathcal B$, the two-sided temperature bound
\eqref{eq:good-two}, \eqref{eq:Zbounds}, and \eqref{eq:Zbar-bounds} give
\begin{equation}\label{eq:pressure-average-lower}
 \frac1{\overline Z_N}\int_N^{N+1}\frac{Z_N\theta}{v}\dd y
 \geq c_0>0.
\end{equation}
The definition of $L_N$ and \eqref{eq:Zbounds} also give
\begin{equation}\label{eq:pressure-at-max}
 \frac{\theta(x_+,t)}{v(x_+,t)}\leq C\mathrm e^{-L_N(t)}.
\end{equation}
Choose a fixed $L_*>0$ so that
$C\mathrm e^{-L_N(t)}$ is at
most $c_0/2$ whenever $L_N>L_*$.  Moreover, the
fundamental theorem of calculus and \eqref{eq:Zdef} give, for every
$x\in[N,N+1]$,
\[
 \frac12|\mathbf{b}(x,t)|^2
 -\frac1{2\overline Z_N(t)}\int_N^{N+1}Z_N(y,t)|\mathbf{b}(y,t)|^2\dd y
 \leq\int_N^{N+1}|\mathbf{b}||\mathbf{b}_y|\dd y.
\]
By the final assertion of Lemma~\ref{l:Kaz}, the map
$t\mapsto\log(Z_N/\overline Z_N)$ is absolutely continuous with values
in $C([N,N+1])$.  Therefore $L_N$ is absolutely continuous on $[0,T]$,
and, for almost every $t$,
\begin{equation}\label{eq:max-envelope}
 L_N'(t)\leq
 \max_{x\in\operatorname{Argmax}_{[N,N+1]}R_N(\cdot,t)}
 \partial_tR_N(x,t).
\end{equation}
Applying \eqref{eq:Z-evol} in \eqref{eq:max-envelope}, we
obtain, whenever $L_N>L_*$,
\begin{equation}\label{eq:LN-differential}
 L_N'(t)\leq
 \begin{cases}
 -c+\displaystyle\int_N^{N+1}|\mathbf{b}||\mathbf{b}_y|\dd y,&t\notin\mathcal B,\\[2mm]
 C(1+D_\theta(t))+\displaystyle\int_N^{N+1}|\mathbf{b}||\mathbf{b}_y|\dd y,&t\in\mathcal B.
 \end{cases}
\end{equation}
Here the first line uses the strict pressure gap $c_0/2$, while the
second uses \eqref{eq:max-pressure} and the nonnegativity of the
weighted pressure average.

If $L_N(t)>\max\{L_*,L_N(0)\}$, let $s$ be the left
endpoint of the corresponding connected component.  The absolute
continuity just established gives
$L_N(s)=\max\{L_*,L_N(0)\}\leq C$.  Integrating the
differential inequality over
$(s,t)\setminus\mathcal B$ and $(s,t)\cap\mathcal B$, and using
\eqref{eq:bad-measure-up}, we obtain
\begin{align}
 L_N(t)&\leq L_N(s)-c|(s,t)\setminus\mathcal B|
 +C\int_{(s,t)\cap\mathcal B}(1+D_\theta(r))\dd r\notag\\
 &\quad+\int_s^t\!\int_N^{N+1}|\mathbf{b}||\mathbf{b}_y|\dd y\dd r\notag\\
 &\leq C-c(t-s)
 +\int_s^t\!\int_N^{N+1}|\mathbf{b}||\mathbf{b}_y|\dd y\dd r.
\end{align}
The oscillation estimate \eqref{eq:b-osc} now yields
\[
 L_N(t)\leq C+\sup_{\tau\geq0}
 \{-c\tau+C\sqrt{1+\tau}\}\leq C.
\]
Outside these components, $L_N$ is bounded by the defining threshold;
also $L_N(0)\leq C$ by the initial bounds. Thus $L_N\leq C$ uniformly
in $N$ and $T$. By \eqref{eq:Zbar-bounds}, $\overline Z_N\leq C$, so
$Z_N\leq C$, and \eqref{eq:Zmax-v} gives \eqref{eq:v-up}.
\end{proof}

The estimates in Propositions~\ref{p:vlow}, \ref{p:tlow},
\ref{p:bcap}, and \ref{p:v-up} are independent of the temporary bootstrap
size $M$.  We now fix the parameters in the order required by the
absorption argument.  First choose $\delta_*$ as in
Proposition~\ref{p:bcap}.  Under \eqref{eq:boot}--\eqref{eq:alpha-small},
let $c_*,C_*>0$ be the bounds furnished by
Propositions~\ref{p:vlow} and \ref{p:v-up}, respectively, and use the
positive lower bound for $\theta$ furnished by Proposition~\ref{p:tlow}.
These constants depend only on the data and $\beta$, and not on
$M$, $T$, or $\alpha$.

Choose $M_*>e$ sufficiently large that
\[
[\inf v_0,\sup v_0]\subset(2/M_*,M_*/2),\qquad
c_*>2/M_*,\qquad C_*<M_*/2.
\]
Set $L_*=\log M_*$.  To ensure both smallness conditions in
\eqref{eq:alpha-small}, define
\begin{equation}\label{eq:alpha-choice}
\alpha_0=\min\left\{1,\frac{\delta_*}{2L_*},
\frac1{2L_*^2}\right\}>0.
\end{equation}
For $0\leq\alpha\leq\alpha_0$, the initial bounds and continuity give
$M_*^{-1}<v<M_*$ on a nonempty initial interval.  On every interval on
which this bootstrap holds, Propositions~\ref{p:vlow} and \ref{p:v-up}
improve it to $2/M_*<v<M_*/2$.  Hence $v$ cannot reach either bootstrap
boundary at a finite first exit time before $T_*$.  Consequently, on
the maximal interval of existence,
\begin{subequations}\label{eq:state}
\begin{equation}\label{eq:state-v}
 c\leq v\leq C.
\end{equation}
\begin{equation}\label{eq:state-theta}
 \theta\geq c.
\end{equation}
\begin{equation}\label{eq:state-mu}
 c\leq\mu\leq C.
\end{equation}
\begin{equation}\label{eq:state-mu-derivative}
 |\mu'(v)|\leq C.
\end{equation}
\end{subequations}
The estimate for $\mu'$ follows from
$\mu'(v)=-\alpha v^{-\alpha-1}$, $0\leq\alpha\leq1$, and the improved
state interval.  Since the bootstrap uses no derivative norm, the
higher-order estimates needed for continuation may now be derived from
\eqref{eq:state}.

The primitive comparison is uniform whenever
\eqref{eq:alpha-small} holds. The lower bounds for $v$ and $\theta$ determine the
constants in the elliptic estimate; Proposition~\ref{p:bcap} then fixes
$\delta_*$, and Proposition~\ref{p:v-up} fixes $C_*$.  Only afterward
do we choose $M_*$ and $\alpha_0$.  In particular, the viscosity error
is absorbed through the integrals in \eqref{eq:heat-conv},
\eqref{eq:HK-eta}, and \eqref{eq:new-error} before any estimate of
$v_x$ is used.

\section{Uniform estimates for higher derivatives}\label{sec:high}
In view of \eqref{eq:state}, all diffusion coefficients are bounded
away from zero. We first estimate $\mathbf{w}_x/v$ and $\mathbf{b}_x/v$. A
separate superlevel estimate is needed when $\beta=0$; it then yields
the bound for $v_x$. The estimates for $u$ and $\theta$ complete the
higher-order bounds.
\begin{proposition}\label{p:cap}
For every fixed $\beta\geq0$,
\begin{equation}
\sup_{0\leq t\leq T}\int_\Omega
 \{u^2+|\mathbf{w}|^2+ |\mathbf{b}|^2 \}\dd x+ \int_0^T\!\int_\Omega
 \{\mu u_x^2+|\mathbf{w}_x|^2+|\mathbf{b}_x|^2\}\dd x\dd t\leq C.
                                                               \label{eq:cap}
\end{equation}
If $\beta>0$, then also
\begin{align}
 &\int_0^T\!\int_\Omega
 \frac{\theta_x^2}{\theta}\dd x\dd t\leq C.
                                                               \label{eq:cappos}
\end{align}
\end{proposition}

\begin{proof}
Decreasing the constant $c$ in \eqref{eq:state}, if necessary, we may
assume that $c\leq1$. Fix $K\geq4$, to be chosen below, and define
on $[c,\infty)$
\begin{equation}
 g_\beta(\xi)=
 \begin{cases}
 \displaystyle 2(1-\xi^{-1}),&c\leq\xi\leq K,\\[2mm]
 \displaystyle
 2-\frac1K-
 \frac{\beta\left(\frac{\xi}{K}\right)^{-\beta}
 +\left(\frac{\xi}{K}\right)^{-\frac12}}
 {K(1+\beta)},&\xi>K.
 \end{cases}                                         \label{eq:capdef}
\end{equation}
Extend $g_\beta$ to $(0,c)$ as a bounded nondecreasing Lipschitz
function, matching its value at $c$. Thus $g_\beta$ is bounded,
continuous, and nondecreasing on $(0,\infty)$, with $g_\beta(1)=0$. Set
\begin{equation}\label{eq:Fbeta-def}
 F_\beta(\xi)=\int_1^\xi g_\beta(s)\dd s.
\end{equation}
Define
\begin{equation}\label{eq:zeta-def}
 \zeta_\beta=2-g_\beta.
\end{equation}
For every $\xi\geq c$,
\begin{equation}\label{eq:Fbeta-bound}
 0\leq F_\beta(\xi)\leq C\Phi(\xi).
\end{equation}
Moreover,
\begin{equation}\label{eq:lj03}
 \zeta_\beta(\xi)\geq K^{-1}.
\end{equation}

The formal multiplications and integrations by parts in this proof are justified by the approximation argument in Appendix~\ref{app:energy}.

Multiplying \eqref{eq:mom}, \eqref{eq:trans}, and \eqref{eq:ind} by
$u$, $\mathbf{w}$, and $\mathbf{b}$, respectively, and adding the resulting identities
to \eqref{eq:mass} multiplied by $1-v^{-1}$ give
\begin{equation}
 \frac{\dd}{\dd t}\int_\Omega\left\{
 \frac{u^2+|\mathbf{w}|^2+v|\mathbf{b}|^2}{2}+\Phi(v)\right\}\dd x
 +\int_\Omega\frac{\mu u_x^2+|\mathbf{w}_x|^2+|\mathbf{b}_x|^2}{v}\dd x
 =\int_\Omega\frac{(\theta-1)u_x}{v}\dd x.           \label{eq:mech-v}
\end{equation}
Set $\tau=\theta/K$. Multiplying \eqref{eq:temp} by
$g_\beta(\theta)$, adding twice \eqref{eq:mech-v}, and integrating by
parts, we obtain
\begin{align}
 &\frac{\dd}{\dd t}\int_\Omega
 \{u^2+|\mathbf{w}|^2+v|\mathbf{b}|^2+2\Phi(v)+F_\beta(\theta)\}\dd x\notag\\
 &\quad+\int_\Omega\frac{
 \zeta_\beta(\theta)(\mu u_x^2+|\mathbf{w}_x|^2+|\mathbf{b}_x|^2)
 +g_\beta'(\theta)\theta^\beta\theta_x^2}{v}\dd x\notag\\
 &=\int_\Omega
 \frac{\{2(\theta-1)-g_\beta(\theta)\theta\}u_x}{v}\dd x\notag\\&\leq\frac12\int_\Omega
 \frac{\zeta_\beta(\theta)\mu u_x^2}{v}\dd x
 +C_\beta K\int_\Omega(\tau-1)_+^2 \dd x,
                                                               \label{eq:cap-id}
\end{align}
where the boundary terms vanish by \eqref{eq:bc-C}--\eqref{eq:bc-D}:
at $x=0$, $\theta_x=0$ and $u=\mathbf{w}=0$, while the magnetic term vanishes
under either condition on $\mathbf{b}$. The estimate of the right-hand side in
\eqref{eq:cap-id} follows from
\eqref{eq:Fbeta-bound}, \eqref{eq:lj03}, and
$\left|2(\theta-1)-g_\beta(\theta)\theta\right|
\leq C_\beta(\tau-1)_+$ for $\theta\geq c$.
Until $K$ is fixed below, the generic constants are independent of
$K$ and time and have the dependence specified in
Section~\ref{sec:lower}.

The superlevel term in \eqref{eq:cap-id} satisfies
\begin{align}
 C_\beta K\int_\Omega(\tau-1)_+^2 \dd x
 &\leq C_\beta K
 \norm{(\sqrt\tau-1)_+}_{L^\infty(\Omega)}^2
 \int_{\{\theta>K\}} \tau  \dd x\notag\\&\leq C_\beta
 \norm{(\sqrt\tau-1)_+}_{L^\infty(\Omega)}^2
 \notag\\ &\leq C_\beta\left(\int_{\{\theta>K\}} \tau^{\frac14}
  |(\tau^{\frac14})_x|  \dd x\right)^2\notag\\ &\leq C_\beta \int_{\{\theta>K\}} \tau^{\frac12}\dd x \int_{\{\theta>K\}}
 |(\tau^{\frac14})_x|^2 \dd x\notag\\ &\leq C_\beta \int_{\{\theta>4\}} \tau\dd x \int_{\{\theta>K\}}
 \tau^{-\frac32} \tau_x^2 \dd x \notag\\
 &\leq\frac{C_\beta}{K}
 \int_{\{\theta>K\}}
 \tau^{\beta-\frac32}\frac{\tau_x^2}{v}\dd x\notag\\
 &\leq \frac{C_\beta}{K}\int_\Omega
 \frac{g_\beta'(\theta)\theta^\beta\theta_x^2}{v}\dd x,
                                                               \label{eq:hotcap}
\end{align}
where the final step in \eqref{eq:hotcap} follows from
\begin{equation}\label{eq:gbeta-id}
 \frac{g_\beta'(\theta)\theta^\beta\theta_x^2}{v}
 =\frac{K^\beta}{1+\beta}
 \left\{\beta^2\tau^{-1}
 +\frac12\tau^{\beta-\frac32}\right\}
 \frac{\tau_x^2}{v}.
\end{equation}

Choose $K=\max\{4C_\beta,4\}$. Substitution of \eqref{eq:hotcap} into
\eqref{eq:cap-id} allows
$C_\beta K\int_\Omega(\tau-1)_+^2\dd x$ to be absorbed into the
$g_\beta'$-dissipation. Integration in time then gives
\begin{align*}
 &\sup_{0\leq t\leq T}\int_\Omega
 \{u^2+|\mathbf{w}|^2+v|\mathbf{b}|^2+2\Phi(v)+F_\beta(\theta)\}\dd x\\
 &\quad+\int_0^T\!\int_\Omega
 \frac{\zeta_\beta(\theta)(\mu u_x^2+|\mathbf{w}_x|^2+|\mathbf{b}_x|^2)
 +g_\beta'(\theta)\theta^\beta\theta_x^2}{v}\dd x\dd t
 \leq C.
\end{align*}
Since $\zeta_\beta(\theta)\geq K^{-1}$ and \eqref{eq:state} gives finite
upper and positive lower bounds for $v$, this proves \eqref{eq:cap}.

For $\beta>0$ and $c\leq\theta\leq K$, we have
\[
 g_\beta'(\theta)\theta^\beta
 =2\theta^{\beta-2}
 \geq\frac{c_{\beta,K}}{\theta}.
\]
For $\theta>K$, we use $\tau=\frac{\theta}{K}$ to obtain
\[
 g_\beta'(\theta)\theta^\beta
 =\frac{K^{\beta-2}}{1+\beta}
 \left\{\beta^2\tau^{-1}
 +\frac12\tau^{\beta-\frac32}\right\}
 \geq\frac{\beta^2K^{\beta-1}}{1+\beta}\frac1\theta.
\]
Therefore,
\begin{equation}\label{eq:gbeta-lower}
 g_\beta'(\theta)\theta^\beta
 \geq\frac{c_{\beta,K}}{\theta},
 \qquad \theta\geq c.
\end{equation}
Combining \eqref{eq:gbeta-lower} with \eqref{eq:state} proves
\eqref{eq:cappos}.
\end{proof}

\begin{lemma}\label{l:flux}
We have
\begin{align}
 &\sup_{0\leq t\leq T}
 (\|\mathbf{w}\|_{H^1}^2
 + \|\mathbf{b}\|_{H^1}^2 )  +\int_0^T\!\int_\Omega\biggl\{
 \left|\mathbf{w}_t \right|^2
 + \left|\mathbf{b}_t \right|^2
 +\left|\left(\frac{\mathbf{w}_x}{v}\right)_x\right|^2
 +\left|\left(\frac{\mathbf{b}_x}{v}\right)_x\right|^2
 \biggr\}\dd x\dd t\leq C.                         \label{eq:trflux}
\end{align}
Moreover,
\begin{align}
 &\int_0^T\bigl\{
 \norm{\mathbf{w}_x(t)}_{L^\infty}^2+\norm{\mathbf{b}_x(t)}_{L^\infty}^2
 +\norm{\mathbf{w}_x(t)}_{L^\infty}^4+\norm{\mathbf{b}_x(t)}_{L^\infty}^4
  \bigr\}\dd t\leq C.
                                                               \label{eq:tr-max}
\end{align}
\end{lemma}

\begin{proof}
Set
\begin{equation}\label{eq:rw-def}
 r_w=\frac{\mathbf{w}_x}{v}.
\end{equation}
Define likewise
\begin{equation}\label{eq:rb-def}
 r_b=\frac{\mathbf{b}_x}{v}.
\end{equation}
Equations \eqref{eq:trans} and \eqref{eq:ind} can be written as
\begin{equation}\label{eq:rw-id}
 (r_w)_x=(\mathbf{w}_t-ur_w)+ur_w-vr_b.
\end{equation}
The induction equation becomes
\begin{equation}\label{eq:rb-id}
 (r_b)_x=v\left\{(\mathbf{b}_t-ur_b)+ur_b+\frac{u_x}{v}\mathbf{b}-r_w\right\}.
\end{equation}

We first estimate $r_w$. Taking the inner product of \eqref{eq:rw-id}
with $\mathbf{w}_t-ur_w$ and integrating over $\Omega$, we use $\mathbf{w}_x=vr_w$ and
$v_t=u_x$ to obtain
\begin{align*}
 \int_\Omega (r_w)_x\cdot(\mathbf{w}_t-ur_w)\dd x
 &=-\int_\Omega r_w\cdot(\mathbf{w}_x)_t\dd x
 +\frac12\int_\Omega u_x|r_w|^2\dd x\\
 &=-\int_\Omega v r_w\cdot (r_w)_t\dd x
 -\frac12\int_\Omega u_x|r_w|^2\dd x\\
 &=-\frac12\frac{\dd}{\dd t}\int_\Omega v|r_w|^2\dd x.
\end{align*}
Thus,
\begin{equation}
 \frac12\frac{\dd}{\dd t}\int_\Omega v|r_w|^2\dd x
 +\int_\Omega|\mathbf{w}_t-ur_w|^2\dd x
 =\int_\Omega(vr_b-ur_w)\cdot(\mathbf{w}_t-ur_w)\dd x.       \label{eq:rE}
\end{equation}
For $r_b$, taking the inner product of \eqref{eq:rb-id} with
$\mathbf{b}_t-ur_b$ and using $\mathbf{b}_x=vr_b$ and $v_t=u_x$
gives
\begin{equation}
 \frac12\frac{\dd}{\dd t}\int_\Omega v|r_b|^2\dd x
 +\int_\Omega v|\mathbf{b}_t-ur_b|^2\dd x
 =-\int_\Omega v\left(ur_b+\frac{u_x}{v}\mathbf{b}-r_w\right)
 \cdot(\mathbf{b}_t-ur_b)\dd x.                                \label{eq:sE}
\end{equation}
In the half-line cases, the boundary terms vanish because $u=\mathbf{w}=0$ at
$x=0$ and either $r_b=\mathbf{b}_x/v=0$ or $\mathbf{b}_t=0$ there, according as
\eqref{eq:bc-N} or \eqref{eq:bc-D} is imposed. The whole-line boundary
terms vanish by the far-field conditions.
Define
\begin{equation}
 Y(t)=\frac12\int_\Omega v
 (|r_w|^2+|r_b|^2)\dd x.                              \label{eq:tr-Y}
\end{equation}
Adding \eqref{eq:rE} and \eqref{eq:sE}, and then applying Young's
inequality and \eqref{eq:state}, we absorb half of the two dissipation
terms and find
\begin{align}
 Y'(t)&+c\int_\Omega\{|\mathbf{w}_t-ur_w|^2+v|\mathbf{b}_t-ur_b|^2\}\dd x\notag\\
 &\leq C
\int_\Omega\left(
 |u|^2(|r_w|^2+|r_b|^2)+|u_x|^2|\mathbf{b}|^2+|r_w|^2+|r_b|^2 \right)
 \dd x\notag\\
 &\leq C Y(t)+C\norm{u(t)}_{L^\infty}^2Y(t)
 +C\norm{\mathbf{b}(t)}_{L^\infty}^2
 \int_\Omega\frac{u_x^2}{v}\dd x \notag\\
 &\leq C\bigl(Y(t)+\|u_x(t)\|_{L^2}^2\bigr)+CY(t)\|u_x(t)\|_{L^2}^2.
                                                               \label{eq:flux}
\end{align}

By \eqref{eq:cap} and \eqref{eq:state},
\begin{equation}
 \int_0^T\left( Y(t)+\|u_x\|_{L^2}^2\right)\dd t\leq C.                \label{eq:fluxL1}
\end{equation}
Rewrite \eqref{eq:flux} as
\begin{equation}\label{eq:flux-gronwall}
 Y'(t)\leq C\norm{u_x(t)}_{L^2}^2Y(t)
 +C\bigl(Y(t)+\norm{u_x(t)}_{L^2}^2\bigr).
\end{equation}
The coefficient of $Y$ and the forcing term in \eqref{eq:flux-gronwall}
are integrable uniformly in $T$ by \eqref{eq:fluxL1}. The integral form
of Gronwall's inequality
therefore yields
 \begin{equation}\sup_{0\leq t\leq T}Y(t)+
 \int_0^T\!\int_\Omega
 \{|\mathbf{w}_t-ur_w|^2+v|\mathbf{b}_t-ur_b|^2\}\dd x\dd t\leq C.
                                                               \label{eq:trt}
\end{equation}
By \eqref{eq:GN-infty}, \eqref{eq:fluxL1}, and \eqref{eq:trt},
\begin{equation}\label{eq:flux-mixed-u}
 \int_0^T\!\int_\Omega |u|^2(|r_w|^2+|r_b|^2)\dd x\dd t
 \leq C\int_0^T\bigl(Y+Y\norm{u_x}_{L^2}^2\bigr)\dd t\leq C.
\end{equation}
The bounds in \eqref{eq:trt} and the one-dimensional embedding imply
$\sup_{0\leq t\leq T}\norm{\mathbf{b}}_{L^\infty}\leq C$. Hence
\begin{equation}\label{eq:flux-mixed-b}
 \int_0^T\!\int_\Omega |u_x|^2|\mathbf{b}|^2\dd x\dd t
 +\int_0^T\!\int_\Omega(|r_w|^2+|r_b|^2)\dd x\dd t
 \leq C.
\end{equation}
Substitution into \eqref{eq:rw-id} and \eqref{eq:rb-id}, together with
\eqref{eq:trt},
gives
\begin{equation}\label{eq:flux-derivatives}
 \int_0^T\!\int_\Omega
 \bigl(|\mathbf{w}_t|^2+|\mathbf{b}_t|^2+|(r_w)_x|^2+|(r_b)_x|^2\bigr)\dd x\dd t\leq C.
\end{equation}
Together with \eqref{eq:tr-Y} and \eqref{eq:cap},
\eqref{eq:flux-derivatives} proves \eqref{eq:trflux}.

To prove \eqref{eq:tr-max}, apply the one-dimensional inequality
$\norm{f}_{L^\infty}^2\leq2\norm{f}_{L^2}\norm{f_x}_{L^2}$,
\eqref{eq:fluxL1}, \eqref{eq:trt}, and the $(r_w)_x,(r_b)_x$ bounds in
\eqref{eq:trflux} to obtain
\begin{equation}\label{eq:flux-Linfty}
 \int_0^T\{\norm{r_w(t)}_{L^\infty}^2+\norm{r_b(t)}_{L^\infty}^2
 +\norm{r_w(t)}_{L^\infty}^4+\norm{r_b(t)}_{L^\infty}^4\}\dd t\leq C.
\end{equation}
Since $\mathbf{w}_x=vr_w$, $\mathbf{b}_x=vr_b$, and $v\leq C$, \eqref{eq:flux-Linfty} proves
\eqref{eq:tr-max}.

We shall also use the following consequence of \eqref{eq:GN},
\eqref{eq:cap}, and \eqref{eq:trflux}:
\begin{equation}\label{eq:b-inf}
 \int_0^T\norm{\mathbf{b}(t)}_{L^\infty}^4\dd t
 \leq C\sup_{0\leq t\leq T}\norm{\mathbf{b}(t)}_{L^2}^2
 \int_0^T\norm{\mathbf{b}_x(t)}_{L^2}^2\dd t\leq C.
\end{equation}
\begin{equation}\label{eq:b-inf-prod}
 \int_0^T\!\int_\Omega |\mathbf{b}|^2|\mathbf{b}_x|^2\dd x\dd t
 \leq \sup_{0\leq t\leq T}\norm{\mathbf{b}(t)}_{L^\infty}^2
 \int_0^T\norm{\mathbf{b}_x(t)}_{L^2}^2\dd t\leq C.
\end{equation}
\qedhere
\end{proof}

When $\beta=0$, the entropy dissipation does not directly control
$\theta_x^2$. The following superlevel estimate supplies the missing
bound.

\begin{proposition}\label{p:hot0}
If $\beta=0$, then
\begin{align}
 &\sup_{0\leq t\leq T}\int_\Omega(\theta-2)_+^2\dd x
 +\int_0^T\!\int_\Omega\theta_x^2\dd x\dd t+\int_0^T\norm{(\theta(t)-2)_+}_{L^\infty}^2\dd t\leq C.
                                                               \label{eq:hot-0}
\end{align}
\end{proposition}

\begin{proof}
Testing \eqref{eq:temp} with $(\theta-2)_+$
(cf.\ \cite{LiLiang2016}) and integrating over $\Omega$, we obtain
\begin{align}
 &\frac12\frac{\dd}{\dd t}\int_\Omega(\theta-2)_+^2\dd x
 +\int_{\{\theta>2\}}\frac{\theta_x^2}{v}\dd x\notag\\
 &\quad=-\int_\Omega\frac{\theta u_x}{v}(\theta-2)_+\dd x
 +\int_\Omega\frac{\mu u_x^2}{v}(\theta-2)_+\dd x
 +\int_\Omega\frac{(\theta-2)_+(|\mathbf{w}_x|^2+|\mathbf{b}_x|^2)}{v}\dd x.
                                                        \label{eq:heat-0}
\end{align}
All boundary terms vanish under \eqref{eq:bc-C}--\eqref{eq:bc-D}; in
particular, $\theta_x=0$ at $x=0$ in both half-line problems.

To couple \eqref{eq:heat-0} with the longitudinal momentum equation, test
\eqref{eq:mom} with $2u(\theta-2)_+$. Since
\begin{equation}\label{eq:cutoff-product}
\begin{aligned}
 \partial_t(\theta-2)_+
 &=\mathbf1_{\{\theta>2\}}\theta_t,\\
 2u(\theta-2)_+u_t
 &=\bigl[u^2(\theta-2)_+\bigr]_t
 -\mathbf1_{\{\theta>2\}}u^2\theta_t
\end{aligned}
\end{equation}
hold almost everywhere, integration by parts gives
\begin{equation}\label{eq:u-theta-0}
\begin{aligned}
 &\frac{\dd}{\dd t}\int_\Omega u^2(\theta-2)_+\dd x
 +2\int_\Omega\frac{\mu(\theta-2)_+u_x^2}{v}\dd x\\
 &=\int_{\{\theta>2\}}u^2\theta_t\dd x
 -2\int_{\{\theta>2\}}\frac{\mu uu_x\theta_x}{v}\dd x\\
 &\quad+2\int_\Omega\frac{\theta u_x}{v}(\theta-2)_+\dd x
 +2\int_{\{\theta>2\}}\frac{\theta u\theta_x}{v}\dd x\\
 &\quad+2\int_\Omega\frac{|\mathbf{b}|^2}{2}(\theta-2)_+u_x\dd x
 +2\int_{\{\theta>2\}}\frac{|\mathbf{b}|^2}{2}u\theta_x\dd x\\
 &\triangleq \sum_{i=1}^6 I_i(t).
\end{aligned}
\end{equation}
We first treat $I_1$, which contains $\theta_t$. For $\beta=0$,
\eqref{eq:temp} reads
\[
 \theta_t=\left(\frac{\theta_x}{v}\right)_x
 -\frac{\theta u_x}{v}
 +\frac{\mu u_x^2+|\mathbf{w}_x|^2+|\mathbf{b}_x|^2}{v}.
\]
For $\rho>0$, define
\[
 \varphi_\rho(\xi)=
 \begin{cases}
  0,&\xi\leq2,\\[1mm]
  \frac{\xi-2}{\rho},&2<\xi<2+\rho,\\[2mm]
  1,&\xi\geq2+\rho.
 \end{cases}
\]
With the Lipschitz approximation $\varphi_\rho$, for every
$\varepsilon>0$ we have
\begin{align*}
 I_1(t)
 &=\lim_{\rho\to0^+}\int_\Omega\varphi_\rho(\theta)u^2
 \left(\frac{\theta_x}{v}\right)_x\dd x
 -\int_{\{\theta>2\}}\frac{u^2\theta u_x}{v}\dd x+\int_{\{\theta>2\}}
 \frac{u^2(\mu u_x^2+|\mathbf{w}_x|^2+|\mathbf{b}_x|^2)}{v}\dd x\\
 &=\lim_{\rho\to0^+}\Biggl\{-2\int_\Omega\varphi_\rho(\theta)
 \frac{uu_x\theta_x}{v}\dd x-\int_\Omega\varphi_\rho'(\theta)
 \frac{u^2\theta_x^2}{v}\dd x
\Biggr\}\\
&\quad-\int_{\{\theta>2\}}\frac{u^2\theta u_x}{v}\dd x
 +\int_{\{\theta>2\}}
 \frac{u^2(\mu u_x^2+|\mathbf{w}_x|^2+|\mathbf{b}_x|^2)}{v}\dd x\\
 &\leq-2\int_{\{\theta>2\}}\frac{uu_x\theta_x}{v}\dd x
 -\int_{\{\theta>2\}}\frac{u^2\theta u_x}{v}\dd x+\int_{\{\theta>2\}}
 \frac{u^2(\mu u_x^2+|\mathbf{w}_x|^2+|\mathbf{b}_x|^2)}{v}\dd x\\
 &\leq\varepsilon\int_\Omega\theta_x^2\dd x
 +C_\varepsilon\int_\Omega u^2u_x^2\dd x
 +C_\varepsilon
 \norm{\left(\theta(t)-\frac32\right)_+}_{L^\infty}^2+\int_{\{\theta>2\}}
 \frac{u^2(|\mathbf{w}_x|^2+|\mathbf{b}_x|^2)}{v}\dd x.
\end{align*}
Here $\varphi_\rho'\geq0$ and
$\varphi_\rho(\theta)\to\mathbf1_{\{\theta>2\}}$; the limit follows
by dominated convergence at the level of the regularized solution.
The final estimate of $I_1(t)$ uses
\[
 \theta\leq4\left(\theta-\frac32\right)_+
 \qquad\text{on }\{\theta>2\},
\]
together with \eqref{eq:hot}, \eqref{eq:state},
\eqref{eq:ent-id}, and Young's inequality. The terms $I_2,I_3,I_4$
satisfy
\begin{align*}
 |I_2(t)|
 &\leq\varepsilon\int_\Omega\theta_x^2\dd x
 +C_\varepsilon\int_\Omega u^2u_x^2\dd x,\\
 |I_3(t)|
 &\leq\frac12\int_\Omega\frac{\mu(\theta-2)_+u_x^2}{v}\dd x
 +C\norm{\left(\theta(t)-\frac32\right)_+}_{L^\infty}^2,\\
 |I_4(t)|
 &\leq\varepsilon\int_\Omega\theta_x^2\dd x
 +C_\varepsilon
 \norm{\left(\theta(t)-\frac32\right)_+}_{L^\infty}^2.
\end{align*}
We retain $I_5$ and $I_6$ until after time integration.

The remaining coercive term is obtained by testing \eqref{eq:mom}
with $u^3$:
\begin{align}
 &\frac14\frac{\dd}{\dd t}\int_\Omega u^4\dd x
 +3\int_\Omega\frac{\mu u^2u_x^2}{v}\dd x=3\int_\Omega\frac{\theta}{v}u^2u_x\dd x
 +3\int_\Omega\frac{|\mathbf{b}|^2}{2}u^2u_x\dd x.        \label{eq:u3-0}
\end{align}
Using $\int_\Omega u^2u_x\dd x=0$, we estimate the first term on
the right-hand side of \eqref{eq:u3-0} as follows:
\begin{align}
 \left|\int_\Omega\frac{\theta}{v}u^2u_x\dd x\right|
 &=\left|\int_\Omega\left(\frac{\theta}{v}-1\right)
 u^2u_x\dd x\right|\notag\\
 &\leq\left|\int_\Omega\left\{
 \frac{1-v}{v}+\mathbf1_{\{\theta\leq2\}}\frac{\theta-1}{v}
 \right\}u^2u_x\dd x\right|+\left|\int_{\{\theta>2\}}
 \frac{\theta-1}{v}u^2u_x\dd x\right|\notag\\
 &\leq\norm{\frac{1-v}{v}
 +\mathbf1_{\{\theta\leq2\}}\frac{\theta-1}{v}}_{L^2}
 \norm{u(t)}_{L^\infty}^2\norm{u_x(t)}_{L^2}+\left|\int_{\{\theta>2\}}
 \frac{\theta-1}{v}u^2u_x\dd x\right|\notag\\
 &\leq C\norm{u_x(t)}_{L^2}^2
 +\varepsilon\int_\Omega u^2u_x^2\dd x
 +C_\varepsilon
 \norm{\left(\theta(t)-\frac32\right)_+}_{L^\infty}^2.
                                                        \label{eq:u3-P-0}
\end{align}
The first equality uses
$\frac{\theta}{v}-1=\frac{1-v}{v}+\frac{\theta-1}{v}$. The inequality
in \eqref{eq:u3-P-0}
follows from Hölder's and Young's inequalities,
\eqref{eq:state}, \eqref{eq:ent-id}, and
\eqref{eq:GN}, together with
\begin{equation}\label{eq:v-Phi-square}
 (v-1)^2\leq C\Phi(v).
\end{equation}
On the two temperature regions, respectively,
\begin{equation}\label{eq:theta-Phi-cold}
 (\theta-1)^2\mathbf1_{\{\theta\leq2\}}\leq C\Phi(\theta)
\end{equation}
and
\begin{equation}\label{eq:theta-hot-cutoff}
 (\theta-1)\mathbf1_{\{\theta>2\}}
 \leq4\left(\theta-\frac32\right)_+.
\end{equation}

Adding \eqref{eq:heat-0}, \eqref{eq:u-theta-0}, and $\Lambda_0$ times
\eqref{eq:u3-0}, and using the
estimates for $I_1,\ldots,I_4$,
\eqref{eq:u3-P-0}, \eqref{eq:cap}, and
\eqref{eq:ent-id}, we first choose $\varepsilon>0$ sufficiently small
and then $\Lambda_0>0$ sufficiently large. This gives
\begin{equation}\label{eq:hot0-IJ}
\begin{aligned}
 &\int_\Omega\left\{\frac12(\theta-2)_+^2
 +u^2(\theta-2)_++\frac{\Lambda_0}{4}u^4\right\}\dd x
 +c\int_0^t\!\int_\Omega
 \{(\theta+u^2)u_x^2+\theta_x^2\}\dd x\dd s\\
 &\leq C+C\int_0^t
 \norm{\left(\theta(s)-\frac32\right)_+}_{L^\infty}^2\dd s+C\int_0^t
 \int_\Omega\frac{(\theta-2)_+(|\mathbf{w}_x|^2+|\mathbf{b}_x|^2)}{v}\dd x\dd s\\&\quad
 +C\int_0^t \int_{\{\theta>2\}}\frac{u^2(|\mathbf{w}_x|^2+|\mathbf{b}_x|^2)}{v}\dd x\dd s+C\int_0^t  \int_\Omega\frac{|\mathbf{b}|^2}{2}
 (\theta-2)_+|u_x|\dd x  \dd s\\&
 \quad+C\int_0^t  \int_{\{\theta>2\}}\frac{|\mathbf{b}|^2}{2}|u\theta_x|\dd x\dd s
 +C\int_0^t  \int_\Omega\frac{|\mathbf{b}|^2}{2}u^2|u_x|\dd x   \dd s\\
 & \triangleq C+C\int_0^t
 \norm{\left(\theta(s)-\frac32\right)_+}_{L^\infty}^2\dd s
 +  \sum_{i=1}^5J_i(t) .
\end{aligned}
\end{equation}
For every $\varepsilon>0$, \eqref{eq:hot}, \eqref{eq:state}, and
\eqref{eq:ent-id} give
\begin{align}
 J_1(t)+J_2(t)
 &\leq C\int_0^t
 \{\norm{\mathbf{w}_x(s)}_{L^\infty}^2+\norm{\mathbf{b}_x(s)}_{L^\infty}^2\}
 \left\{\int_\Omega(\theta-2)_+\dd x
 +\int_{\{\theta>2\}}u^2\dd x\right\}\dd s\notag\\
 &\leq C\int_0^t
 \{\norm{\mathbf{w}_x(s)}_{L^\infty}^2+\norm{\mathbf{b}_x(s)}_{L^\infty}^2\}\dd s.
 \label{eq:J12}
\end{align}
For $J_3$, Young's inequality and the uniform bounds for $v$ yield
\begin{align}
 J_3(t)
 &\leq C\int_0^t\norm{\mathbf{b}(s)}_{L^\infty}^2
 \int_\Omega(\theta-2)_+|u_x|\dd x\dd s\notag\\
 &\leq\varepsilon\int_0^t\!\int_\Omega
 \frac{(\theta-2)_+u_x^2}{v}\dd x\dd s
 +C_\varepsilon\int_0^t\norm{\mathbf{b}(s)}_{L^\infty}^4
 \int_\Omega v(\theta-2)_+\dd x\dd s\notag\\
 &\leq\varepsilon\int_0^t\!\int_\Omega
 \frac{(\theta-2)_+u_x^2}{v}\dd x\dd s
 +C_\varepsilon\int_0^t\norm{\mathbf{b}(s)}_{L^\infty}^4\dd s.
 \label{eq:J3}
\end{align}
The fourth remainder term satisfies
\begin{align}
 J_4(t)
 &\leq C\int_0^t\norm{\mathbf{b}(s)}_{L^\infty}^2
 \int_{\{\theta>2\}}|u\theta_x|\dd x\dd s\notag\\
 &\leq\varepsilon\int_0^t\!\int_\Omega\theta_x^2\dd x\dd s
 +C_\varepsilon\int_0^t\norm{\mathbf{b}(s)}_{L^\infty}^4
 \int_{\{\theta>2\}}u^2\dd x\dd s\notag\\
 &\leq\varepsilon\int_0^t\!\int_\Omega\theta_x^2\dd x\dd s
 +C_\varepsilon\int_0^t\norm{\mathbf{b}(s)}_{L^\infty}^4\dd s.
 \label{eq:J4}
\end{align}
For the last remainder term, Young's inequality gives
\begin{align}
 J_5(t)
 &\leq C\int_0^t\norm{\mathbf{b}(s)}_{L^\infty}^2
 \int_\Omega u^2|u_x|\dd x\dd s\notag\\
 &\leq\varepsilon\int_0^t\!\int_\Omega u^2u_x^2\dd x\dd s
 +C_\varepsilon\int_0^t\norm{\mathbf{b}(s)}_{L^\infty}^4
 \int_\Omega u^2\dd x\dd s\notag\\
 &\leq\varepsilon\int_0^t\!\int_\Omega u^2u_x^2\dd x\dd s
 +C_\varepsilon\int_0^t\norm{\mathbf{b}(s)}_{L^\infty}^4\dd s.
 \label{eq:J5}
\end{align}
Here we also used
\begin{equation}\label{eq:hot0-L1}
 \int_\Omega(\theta-2)_+\dd x+\norm{u(t)}_{L^2}^2\leq C.
\end{equation}
Together with \eqref{eq:tr-max}, \eqref{eq:b-inf} gives
\begin{equation}\label{eq:hot0-coeff}
 \int_0^T \left(
 \norm{\mathbf{w}_x(t)}_{L^\infty}^2+\norm{\mathbf{b}_x(t)}_{L^\infty}^2
 +\norm{\mathbf{b}(t)}_{L^\infty}^4\right)\dd t\le C.
\end{equation}
Substituting \eqref{eq:J12}--\eqref{eq:J5} and
\eqref{eq:hot0-coeff} into \eqref{eq:hot0-IJ}, and choosing
$\varepsilon$ so that the
three coercive terms are absorbed, we obtain
\begin{align}
 &\int_\Omega\left\{\frac12(\theta-2)_+^2
 +u^2(\theta-2)_++\frac{\Lambda_0}{4}u^4\right\}\dd x
 +c\int_0^t\!\int_\Omega
 \{(\theta+u^2)u_x^2+\theta_x^2\}\dd x\dd s\notag\\ &\leq C+C\int_0^t
 \norm{\left(\theta(s)-\frac32\right)_+}_{L^\infty}^2\dd s\notag\\&\leq C+C\int_0^t
 \left(\int_{\{\theta>\frac32\}}|\theta_x|\dd x\right)^2\dd s\notag\\&\leq C+C\int_0^t
 \left(\int_{\{\theta>\frac32\}}v\theta\dd x\right)
 \left(\int_{\{\theta>\frac32\}}\frac{\theta_x^2}{v\theta}\dd x\right)\dd s\notag\\ &\leq C+C\varepsilon\int_0^t\!\int_\Omega\theta_x^2\dd x\dd s+C_\varepsilon\int_0^t
 \int_{\{\theta>\frac32\}}\frac{\theta_x^2}{v\theta^2}\dd x \dd s.
                                                        \label{eq:hotid0}
\end{align}
Decrease $\varepsilon>0$, if necessary, so that
$C\varepsilon\int_0^t\!\int_\Omega\theta_x^2\dd x\dd s$ is absorbed into
the left-hand side. The remaining term is controlled by
\eqref{eq:ent-id}, and \eqref{eq:hot-0} follows.
\end{proof}

We next estimate the first derivative of the specific volume.

\begin{lemma}\label{l:vx}
For every fixed $\beta\geq0$,
\begin{equation}
 \sup_{0\leq t\leq T}\int_\Omega v_x^2\dd x
 +\int_0^T\!\int_\Omega v_x^2\dd x\dd t\leq C.
                                                               \label{eq:vx}
\end{equation}
\end{lemma}

\begin{proof}
Combining \eqref{eq:state}, \eqref{eq:cap}, and
\eqref{eq:b-inf-prod} with
\eqref{eq:cappos} when $\beta>0$, and with \eqref{eq:hot-0} when
$\beta=0$, yields
\begin{align}
 \int_0^T\!\int_\Omega\biggl\{&
 \mu u_x^2+|\mathbf{w}_x|^2+|\mathbf{b}_x|^2
 +\frac{\theta_x^2}{ \theta}+|\mathbf{b}|^2|\mathbf{b}_x|^2\biggr\}
 \dd x\dd t\leq C.                                  \label{eq:prehi}
\end{align}
The Cauchy--Schwarz inequality, \eqref{eq:hot}, and \eqref{eq:state}
give
\[
 \norm{(\theta(t)-2)_+}_{L^\infty}^2
 \leq\left(\int_{\{\theta>2\}}|\theta_x|\dd x\right)^2
 \leq C\int_\Omega\frac{\theta_x^2}{\theta}\dd x.
\]
Integration in time and \eqref{eq:prehi} yield
\begin{equation}
 \int_0^T
 \norm{(\theta(t)-2)_+}_{L^\infty}^2 \dd t\leq C.
                                                        \label{eq:excurs}
\end{equation}

Set
\begin{equation}\label{eq:qdef}
 q=M_\alpha(v)_x=\frac{\mu v_x}{v}.
\end{equation}
Using $u_x=v_t$ in \eqref{eq:mom}, we obtain
\begin{equation}\label{eq:q-evolution}
 q_t-u_t=-\frac{\theta v_x}{v^2}+\frac{\theta_x}{v}+\mathbf{b}\cdot \mathbf{b}_x.
\end{equation}
Multiplying \eqref{eq:q-evolution} by $q$ and integrating over $\Omega$, we
integrate the $u_tq$ term by parts in time. Since
\[
 -\int_\Omega u\,q_t\,\dd x
 =-\int_\Omega u\left(\frac{\mu u_x}{v}\right)_x\,\dd x
 =\int_\Omega\frac{\mu u_x^2}{v}\,\dd x
\]
we obtain
\begin{align}
 &\frac{\dd}{\dd t}\int_\Omega
 \left\{\frac12|q|^2-uq\right\}\dd x
 +\int_\Omega\frac{\theta|q|^2}{\mu v}\dd x\notag\\&=\int_\Omega\left\{
 \frac{\mu u_x^2}{v}+\frac{\theta_x q}{v}
 +q\,\mathbf{b}\cdot \mathbf{b}_x\right\}\dd x\\
 &\leq\frac12\int_\Omega
 \frac{\theta|q|^2}{\mu v}\dd x
 +C\int_\Omega\left\{u_x^2+\frac{\theta_x^2}{v\theta}
 +|\mathbf{b}|^2|\mathbf{b}_x|^2\right\}\dd x, \label{eq:lj09}
\end{align}
The spatial boundary term vanishes because $u=0$ at $x=0$ in the
half-line cases and by the far-field conditions on $\mathbb R$. The
final estimate in \eqref{eq:lj09} follows from
\eqref{eq:state} and Young's inequality.

Moreover, \eqref{eq:ent-id} gives
\[
 \int_\Omega\left\{\frac12|q|^2
 -uq\right\}\dd x
 \geq\frac14\int_\Omega|q|^2\dd x
 -C\int_\Omega u^2\dd x
 \geq\frac14\int_\Omega|q|^2\dd x-C.
\]
Integrating \eqref{eq:lj09} over $(0,t)$ and using this lower bound
together with \eqref{eq:init}, \eqref{eq:state}, and
\eqref{eq:prehi}, we find
\[
 \sup_{0\leq t\leq T}\int_\Omega|q|^2\dd x
 +\int_0^T\!\int_\Omega|q|^2\dd x\dd t\leq C.
\]
Since $v_x=vq/\mu$ and \eqref{eq:state} gives finite upper and positive
lower bounds for both $v$ and $\mu$, this proves \eqref{eq:vx}.

\end{proof}

We now estimate the first and second derivatives of the longitudinal
velocity. The estimate of $\theta_x$ is coupled to this argument and
requires separate treatments of $\beta=0$, $0<\beta<2$, and
$\beta\geq2$.

\begin{proposition}\label{p:long}
For every fixed $\beta\geq0$,
\begin{align}
 &\sup_{0\leq t\leq T}\int_\Omega u_x^2\dd x+\int_0^T\!\int_\Omega
 \{u_{xx}^2+u_t^2+\theta_x^2\}\dd x\dd t\leq C.
                                                               \label{eq:long}
\end{align}
Moreover,
\begin{align}
 \int_0^T\biggl[&
 \norm{u_x(t)}_{L^\infty}^2+\norm{\mathbf{w}_x(t)}_{L^\infty}^2
 +\norm{\mathbf{b}_x(t)}_{L^\infty}^2\notag\\
 &+\{\norm{u_x(t)}_{L^\infty}^2+\norm{\mathbf{w}_x(t)}_{L^\infty}^2
 +\norm{\mathbf{b}_x(t)}_{L^\infty}^2\}^2\biggr]\dd t\leq C.
                                                               \label{eq:gmax}
\end{align}
\end{proposition}

\begin{proof}
Expanding \eqref{eq:mom}, we obtain
\begin{equation}
 u_t=\frac{\mu u_{xx}}{v}+\frac{\mu_xu_x}{v}-\frac{\mu u_xv_x}{v^2}-\frac{\theta_x}{v}
 +\frac{\theta v_x}{v^2}-\mathbf{b}\cdot \mathbf{b}_x.                \label{eq:momexp}
\end{equation}
Multiplying \eqref{eq:momexp} by $-u_{xx}$, integrating over
$\Omega$, and integrating the $u_tu_{xx}$ term by parts give
\begin{equation}\label{eq:ux-energy}
\begin{aligned}
 \frac12\frac{\dd}{\dd t}\int_\Omega u_x^2\dd x
 +\int_\Omega\frac{\mu u_{xx}^2}{v}\dd x=\int_\Omega\left\{
 \frac{\mu u_xv_xu_{xx}}{v^2}+\frac{\theta_xu_{xx}}{v}
 -\frac{\theta v_xu_{xx}}{v^2}
 +(\mathbf{b}\cdot \mathbf{b}_x)u_{xx}-\frac{\mu'(v)v_xu_xu_{xx}}v\right\}\dd x.
\end{aligned}
\end{equation}
The boundary contribution from $u_tu_x$ vanishes: in the half-line
cases, $u(0,t)=0$ implies $u_t(0,t)=0$, while on the whole line the
far-field conditions apply. Since $\mu_x=\mu'(v)v_x$,
\eqref{eq:state}, \cref{l:vx}, and \eqref{eq:GN} imply
\begin{align*}
 &\left|\int_\Omega\frac{\mu u_xv_xu_{xx}}{v^2}\dd x\right|
 +\left|\int_\Omega\frac{\mu'(v)v_xu_xu_{xx}}v\dd x\right|\\
 \leq C\norm{v_x}_{L^2}\norm{u_x}_{L^\infty}\norm{u_{xx}}_{L^2}
 \leq\varepsilon\norm{u_{xx}}_{L^2}^2
       +C_\varepsilon\norm{u_x}_{L^2}^2.
\end{align*}
The remaining terms are estimated, for every $\varepsilon>0$, by
\eqref{eq:state}, \eqref{eq:GN}, and Young's inequality:
\begin{align*}
 \left|\int_\Omega\frac{\theta_xu_{xx}}{v}\dd x\right|
 &\leq\varepsilon\norm{u_{xx}}_{L^2}^2
 +C_\varepsilon\int_\Omega\theta_x^2\dd x,\\
 \left|\int_\Omega\frac{\theta v_xu_{xx}}{v^2}\dd x\right|
 &\leq\varepsilon\norm{u_{xx}}_{L^2}^2
 +C_\varepsilon\int_\Omega\theta^2v_x^2\dd x,\\
 \left|\int_\Omega(\mathbf{b}\cdot \mathbf{b}_x)u_{xx}\dd x\right|
 &\leq \norm{u_{xx}}_{L^2} \norm{\mathbf{b}}_{L^\infty} \norm{\mathbf{b}_{x}}_{L^2}
  \leq\varepsilon\norm{u_{xx}}_{L^2}^2
 +C_\varepsilon\int_\Omega |\mathbf{b}_x|^2\dd x.
\end{align*}
Substitution into the energy identity, followed by the choice of a
sufficiently small $\varepsilon$, yields
\begin{align}
 \frac{\dd}{\dd t}\int_\Omega u_x^2\dd x
 +c\int_\Omega u_{xx}^2\dd x
 &\leq C \norm{u_x(t)}_{L^2}^2
 +C\int_\Omega\theta_x^2\dd x+C\norm{v_x(t)}_{L^2}^2\notag\\
 &\quad+C\norm{(\theta(t)-2)_+}_{L^\infty}^2
 +C\int_\Omega |\mathbf{b}_x|^2\dd x.
                                                        \label{eq:ux-est}
\end{align}
Here $\theta\leq2+(\theta-2)_+$ and \cref{l:vx} give
\[
 \int_\Omega\theta^2v_x^2\dd x \leq C\norm{v_x(t)}_{L^2}^2
 +C\norm{(\theta(t)-2)_+}_{L^\infty}^2\norm{v_x(t)}_{L^2}^2 \leq C\norm{v_x(t)}_{L^2}^2
 +C\norm{(\theta(t)-2)_+}_{L^\infty}^2.
\]

We first prove, for every fixed $\beta\geq0$, that
\begin{align}
 \sup_{0\leq t\leq T}\int_\Omega u_x^2\dd x
 +\int_0^T\!\int_\Omega
 \{u_{xx}^2+\theta_x^2\}\dd x\dd t\leq C. \label{eq:lj10}
\end{align}

\noindent\textit{The cases $\beta=0$ and $\beta\geq2$.}\quad
For $\beta=0$, \eqref{eq:hot-0} gives
$\int_0^T\!\int_\Omega\theta_x^2\dd x\dd t\leq C$.
For $\beta\geq2$, \eqref{eq:state} and the temperature dissipation in
\eqref{eq:ent-id} give
\[
 \int_0^T\!\int_\Omega\theta_x^2\dd x\dd t
 \leq C\int_0^T\!\int_\Omega
 \frac{\theta^{\beta-2}\theta_x^2}{v}\dd x\dd t\leq C.
\]
Integrating \eqref{eq:ux-est} over $(0,t)$ and using
\eqref{eq:prehi}, Lemma~\ref{l:vx}, \eqref{eq:excurs}, and
\eqref{eq:b-inf}, we choose the corresponding $\varepsilon$ so that
the resulting multiple of
$\int_0^t\norm{u_{xx}}_{L^2}^2\dd s$ is absorbed. Hence
\begin{align*}
 \sup_{0\leq t\leq T}\norm{u_x(t)}_{L^2}^2
 +\int_0^T\norm{u_{xx}(t)}_{L^2}^2\dd t\leq C,
 \qquad \beta=0\quad\text{or}\quad\beta\geq2.
\end{align*}

\noindent\textit{The case $0<\beta<2$.}\quad
Here the entropy dissipation controls $\theta_x$ only on
$\{\theta\leq2\}$. Testing \eqref{eq:temp} with $(\theta-2)_+$ and
integrating over $\Omega$ give
\begin{align}
 &\frac12\frac{\dd}{\dd t}\int_\Omega(\theta-2)_+^2\dd x
 +\int_{\{\theta>2\}}\frac{\theta^\beta\theta_x^2}{v}\dd x=-\int_\Omega\frac{\theta u_x}{v}(\theta-2)_+\dd x
 +\int_\Omega\frac{(\theta-2)_+
 (\mu u_x^2+|\mathbf{w}_x|^2+|\mathbf{b}_x|^2)}{v}\dd x.
                                                               \label{eq:therm}
\end{align}
The boundary heat-flux term vanishes as in the proof of
Proposition~\ref{p:hot0}. By \eqref{eq:hot} and \eqref{eq:state},
\[
 \int_\Omega(\theta-2)_+\dd x\leq C.
\]
In addition,
\[
 \int_\Omega\theta(\theta-2)_+\dd x
 \leq C\norm{(\theta(t)-2)_+}_{L^\infty}.
\]
Therefore, Young's inequality and \eqref{eq:GN} give, for every
$\varepsilon>0$,
\begin{align*}
 &\left|\int_\Omega\frac{\theta u_x}{v}(\theta-2)_+\dd x\right|
 +\int_\Omega\frac{(\theta-2)_+
 (\mu u_x^2+|\mathbf{w}_x|^2+|\mathbf{b}_x|^2)}{v}\dd x\\
 &\quad\leq C\{\norm{u_x(t)}_{L^\infty}^2+\norm{\mathbf{w}_x(t)}_{L^\infty}^2
 +\norm{\mathbf{b}_x(t)}_{L^\infty}^2\}
 +C\norm{u_x(t)}_{L^\infty}\norm{(\theta(t)-2)_+}_{L^\infty}\\
 &\quad\leq\varepsilon\norm{u_{xx}(t)}_{L^2}^2
 +C_\varepsilon\norm{u_x(t)}_{L^2}^2
 +C\{\norm{\mathbf{w}_x(t)}_{L^\infty}^2+\norm{\mathbf{b}_x(t)}_{L^\infty}^2\}
 +C\norm{(\theta(t)-2)_+}_{L^\infty}^2.
\end{align*}

Multiply \eqref{eq:therm} by $2\Lambda_1$, where $\Lambda_1$ is fixed
sufficiently large, and add the result to \eqref{eq:ux-est}. Since
$\theta^\beta/v$ is bounded below on $\{\theta>2\}$, choosing
$\varepsilon$ sufficiently small yields
\begin{align}
 &\frac{\dd}{\dd t}\left\{
 \norm{u_x(t)}_{L^2}^2+\Lambda_1\int_\Omega(\theta-2)_+^2\dd x\right\}
 +c\int_\Omega u_{xx}^2\dd x
 +c\int_{\{\theta>2\}}\theta_x^2\dd x\notag\\
 &\quad\leq  C\biggl\{
 \norm{u_x(t)}_{L^2}^2
 +\int_{\{\theta\leq2\}}\theta_x^2\dd x
 + \norm{(\theta(t)-2)_+}_{L^\infty}^2 +\norm{v_x(t)}_{L^2}^2\notag\\
 &\qquad+\norm{\mathbf{b}_x  }_{L^2}^2
 +\norm{\mathbf{w}_x }_{L^\infty}^2+\norm{\mathbf{b}_x }_{L^\infty}^2\biggr\}.
                                                               \label{eq:uxheat}
\end{align}
Integrating \eqref{eq:uxheat} over $(0,t)$ and using
\eqref{eq:prehi}, \eqref{eq:excurs}, \cref{l:vx}, and
\eqref{eq:tr-max}, we absorb the
remaining multiple of $\int_0^t\norm{u_{xx}}_{L^2}^2\dd s$ and obtain
\begin{align*}
 \sup_{0\leq t\leq T}\norm{u_x(t)}_{L^2}^2
 +\int_0^T\norm{u_{xx}(t)}_{L^2}^2\dd t+\int_0^T\!\int_{\{\theta>2\}}
 \theta_x^2\dd x\dd t\leq C.
\end{align*}
On $\{\theta\leq2\}$, \eqref{eq:state} and \eqref{eq:prehi} control
$\theta_x^2$. Thus \eqref{eq:lj10} holds for all three ranges of
$\beta$.

\noindent\textit{The time derivative and pointwise gradients.}\quad
The one-dimensional inequality \eqref{eq:GN-infty} and
\eqref{eq:lj10} give
\begin{equation}\label{eq:ux-inf-int}
 \int_0^T\norm{u_x(t)}_{L^\infty}^2\dd t
 \leq C\left(\int_0^T\norm{u_x(t)}_{L^2}^2\dd t\right)^{\frac12}
 \left(\int_0^T\norm{u_{xx}(t)}_{L^2}^2\dd t\right)^{\frac12}\leq C.
\end{equation}
\begin{equation}\label{eq:ux-inf-four}
 \int_0^T\norm{u_x(t)}_{L^\infty}^4\dd t
 \leq C\sup_{0\leq t\leq T}\norm{u_x(t)}_{L^2}^2
 \int_0^T\norm{u_{xx}(t)}_{L^2}^2\dd t\leq C.
\end{equation}
Since $\mu_x=\mu'(v)v_x$, \eqref{eq:state}, \cref{l:vx}, and
\eqref{eq:ux-inf-int} give
\begin{equation}\label{eq:muux}
 \int_0^T\!\int_\Omega|\mu_xu_x|^2\dd x\dd t
 \leq C\sup_{t\leq T}\norm{v_x}_{L^2}^2
       \int_0^T\norm{u_x}_{L^\infty}^2\dd t\leq C.
\end{equation}
The same bounds give
\begin{equation}\label{eq:ux-vx-product}
 \int_0^T\!\int_\Omega u_x^2v_x^2\dd x\dd t
 \leq\sup_{0\leq t\leq T}\norm{v_x(t)}_{L^2}^2
 \int_0^T\norm{u_x(t)}_{L^\infty}^2\dd t\leq C.
\end{equation}
By \eqref{eq:excurs} and Lemma~\ref{l:vx},
\begin{equation}\label{eq:theta-vx-product}
 \int_0^T\!\int_\Omega\theta^2v_x^2\dd x\dd t
 \leq C\int_0^T\{\norm{v_x(t)}_{L^2}^2
 +\norm{(\theta(t)-2)_+}_{L^\infty}^2\}\dd t\leq C.
\end{equation}
Using \eqref{eq:momexp}, \eqref{eq:muux},
\eqref{eq:ux-vx-product}, \eqref{eq:theta-vx-product},
\eqref{eq:b-inf-prod}, and \eqref{eq:lj10}, we obtain
\begin{equation}\label{eq:ut-bound}
 \int_0^T\!\int_\Omega u_t^2\dd x\dd t
 \leq C\int_0^T\!\int_\Omega
 \{u_{xx}^2+u_x^2v_x^2+\theta_x^2+\theta^2v_x^2
 +|\mathbf{b}|^2|\mathbf{b}_x|^2+|\mu_xu_x|^2\}\dd x\dd t\leq C.
\end{equation}
Equations \eqref{eq:lj10} and \eqref{eq:ut-bound} prove
\eqref{eq:long}. The two bounds for
$\norm{u_x(t)}_{L^\infty}$ in \eqref{eq:ux-inf-int} and
\eqref{eq:ux-inf-four}, combined with
\eqref{eq:tr-max}, prove \eqref{eq:gmax}.
\end{proof}

It remains to derive the uniform upper bound and the second-order
estimates for the temperature.

\begin{proposition}\label{p:thig}
For every fixed $\beta\geq0$,
\begin{equation}\label{eq:t-up}
 \sup_{(x,t)\in\Omega\times[0,T]}\theta(x,t)\leq C.
\end{equation}
Moreover,
\begin{equation}\label{eq:t-H2}
 \sup_{0\leq t\leq T}\int_\Omega\theta_x^2\dd x
 +\int_0^T\!\int_\Omega
 \{\theta_t^2+\theta_{xx}^2\}\dd x\dd t\leq C.
\end{equation}
\end{proposition}

\begin{proof}
\noindent\textit{High-temperature moments.}\quad
If $\beta=0$, \eqref{eq:hot-0} and
\eqref{eq:hot} give
\[
 \sup_{0\leq t\leq T}\int_{\{\theta>2\}}(1+\theta^2)\dd x\leq C.
\]
They also give
\[
 \int_0^T\!\int_{\{\theta>4\}}\theta_x^2\dd x\dd t\leq C.
\]

Let $\beta>0$. Testing \eqref{eq:temp} with
$(\theta-2)_+^{\beta+1}$ (cf.\ \cite{LiXu2026}), and using the chain
rule, we obtain the exact identity
\begin{align}
 &\frac1{\beta+2}\frac{\dd}{\dd t}
 \int_\Omega(\theta-2)_+^{\beta+2}\dd x
 +(\beta+1)\int_{\{\theta>2\}}
 \frac{\theta^\beta(\theta-2)_+^\beta\theta_x^2}{v}\dd x\notag\\
 &\quad=-\int_\Omega\frac{\theta u_x}{v}(\theta-2)_+^{\beta+1}\dd x
 +\int_\Omega\frac{\mu u_x^2+|\mathbf{w}_x|^2+|\mathbf{b}_x|^2}{v}
 (\theta-2)_+^{\beta+1}\dd x.
 \label{eq:ht-identity}
\end{align}
By \eqref{eq:hot} and \eqref{eq:state}, the two terms on the right of
\eqref{eq:ht-identity} satisfy
\begin{equation}\label{eq:ht-pressure}
 \left|\int_\Omega\frac{\theta u_x}{v}
 (\theta-2)_+^{\beta+1}\dd x\right|
 \leq C\norm{u_x(t)}_{L^\infty}
 \norm{(\theta(t)-2)_+}_{L^\infty}^{\beta+1}.
\end{equation}
The production term satisfies
\begin{equation}\label{eq:ht-source}
 \int_\Omega\frac{\mu u_x^2+|\mathbf{w}_x|^2+|\mathbf{b}_x|^2}{v}
 (\theta-2)_+^{\beta+1}\dd x
 \leq C\{\norm{u_x(t)}_{L^\infty}^2
 +\norm{\mathbf{w}_x(t)}_{L^\infty}^2+\norm{\mathbf{b}_x(t)}_{L^\infty}^2\}
 \norm{(\theta(t)-2)_+}_{L^\infty}^{\beta}.
\end{equation}
Combining \eqref{eq:ht-identity}, \eqref{eq:ht-pressure}, and
\eqref{eq:ht-source} gives
\begin{align}
 &\frac1{\beta+2}\frac{\dd}{\dd t}
 \int_\Omega(\theta-2)_+^{\beta+2}\dd x
 +c(\beta+1)\int_{\{\theta>2\}}
 \theta^\beta(\theta-2)_+^\beta\theta_x^2\dd x\notag\\&\leq C\{\norm{u_x(t)}_{L^\infty}^2
 +\norm{\mathbf{w}_x(t)}_{L^\infty}^2+\norm{\mathbf{b}_x(t)}_{L^\infty}^2\}
 \norm{(\theta(t)-2)_+}_{L^\infty}^\beta\notag\\
 &\quad+C\norm{u_x(t)}_{L^\infty}
 \norm{(\theta(t)-2)_+}_{L^\infty}^{\beta+1}.
                                                        \label{eq:httest}
\end{align}
The boundary heat-flux term in \eqref{eq:ht-identity} vanishes under
\eqref{eq:bc-C}--\eqref{eq:bc-D}.

On every connected component of $\{\theta>2\}$, the fundamental
theorem of calculus and the Cauchy--Schwarz inequality give
\begin{align}
 \norm{(\theta(t)-2)_+}_{L^\infty}^{2\beta+3}
 &\leq C\left(\int_{\{\theta>2\}}
 \frac{\theta^\beta(\theta-2)_+^\beta\theta_x^2}{v}\dd x\right)
 \left(\int_{\{\theta>2\}}v(\theta-2)_+^{\beta+1}
 \theta^{-\beta}\dd x\right)
 \notag\\
 &\leq C\int_{\{\theta>2\}}
 \theta^\beta(\theta-2)_+^\beta\theta_x^2\dd x.
                                               \label{eq:htmin}
\end{align}
Here we used
$(\theta-2)_+^{\beta+1}\theta^{-\beta}\leq(\theta-2)_+$,
\eqref{eq:state}, and \eqref{eq:hot}; the far-field condition supplies
an endpoint at which $(\theta-2)_+$ vanishes.

By Young's inequality, for every $\varepsilon>0$,
\begin{subequations}\label{eq:hot-young}
\begin{equation}\label{eq:hot-young-source}
\begin{aligned}
 &\{\norm{u_x(t)}_{L^\infty}^2+\norm{\mathbf{w}_x(t)}_{L^\infty}^2
 +\norm{\mathbf{b}_x(t)}_{L^\infty}^2\}
 \norm{(\theta(t)-2)_+}_{L^\infty}^\beta
 \\
 &\quad\leq\varepsilon
 \norm{(\theta(t)-2)_+}_{L^\infty}^{2\beta+3}
 +C_\varepsilon\{\norm{u_x(t)}_{L^\infty}^2
 +\norm{\mathbf{w}_x(t)}_{L^\infty}^2+\norm{\mathbf{b}_x(t)}_{L^\infty}^2\}^{\frac{2\beta+3}{\beta+3}}.
\end{aligned}
\end{equation}
Moreover,
\begin{equation}\label{eq:hot-young-pressure}
 \norm{u_x(t)}_{L^\infty}
 \norm{(\theta(t)-2)_+}_{L^\infty}^{\beta+1}
 \leq\varepsilon
 \norm{(\theta(t)-2)_+}_{L^\infty}^{2\beta+3}
 +C_\varepsilon\{\norm{u_x(t)}_{L^\infty}^2
 +\norm{(\theta(t)-2)_+}_{L^\infty}^2\}.
\end{equation}
\end{subequations}
For \eqref{eq:hot-young-pressure}, distinguish between
$\norm{(\theta(t)-2)_+}_{L^\infty}\leq1$ and
$\norm{(\theta(t)-2)_+}_{L^\infty}>1$.
Since $1<\frac{2\beta+3}{\beta+3}<2$,
\eqref{eq:gmax} gives
\begin{equation}\label{eq:hot-source-L1}
 \{\norm{u_x(t)}_{L^\infty}^2+\norm{\mathbf{w}_x(t)}_{L^\infty}^2
 +\norm{\mathbf{b}_x(t)}_{L^\infty}^2\}^{\frac{2\beta+3}{\beta+3}}
 \in L^1(0,T).
\end{equation}
Its $L^1(0,T)$ norm and the one in \eqref{eq:excurs} are bounded
independently of $T$. Choose
$\varepsilon$ in \eqref{eq:hot-young} so that, by
\eqref{eq:htmin}, both high-temperature terms are absorbed into the
dissipation in \eqref{eq:httest}. Integration in time and
\eqref{eq:init} then give
\begin{align}
 \sup_{0\leq t\leq T}\int_\Omega(\theta-2)_+^{\beta+2}\dd x+\int_0^T\!\int_{\{\theta>2\}}
 \theta^\beta(\theta-2)_+^\beta\theta_x^2\dd x\dd t\leq C.
                                                               \label{eq:htmom}
\end{align}
On $\{\theta>4\}$,
$\frac{\theta}{2}\leq(\theta-2)_+\leq\theta$.
Together with \eqref{eq:hot} and
\eqref{eq:htmom}, this yields
\begin{equation}\label{eq:hot-Lp}
 \sup_{0\leq t\leq T}\int_{\{\theta>2\}}
 \{\theta^\beta+\theta^{\beta+2}\}\dd x\leq C.
\end{equation}
Moreover,
\begin{equation}\label{eq:htgrad}
 \int_0^T\!\int_{\{\theta>4\}}
 \theta^{2\beta}\theta_x^2\dd x\dd t\leq C.
\end{equation}
For $\beta=0$, the same two conclusions follow from
\eqref{eq:hot-0} and \eqref{eq:hot}. Hence \eqref{eq:hot-Lp} and
\eqref{eq:htgrad} hold for every fixed $\beta\geq0$.

\noindent\textit{The weighted time derivative.}\quad
Set
\[
 Y_\beta(t)=\int_\Omega
 \frac{\theta^{2\beta}\theta_x^2}{v}\dd x.
\]
Testing \eqref{eq:temp} with $\theta^\beta\theta_t$
(cf.\ \cite{CaoPengSun2021,TongWangZhang2026} and, for $\beta=0$,
\cite{LuShiXiong2021}) and using
\[
 (\theta^\beta\theta_t)_x=(\theta^\beta\theta_x)_t,
 \qquad v_t=u_x,
\]
give
\begin{align}
 \frac12Y_\beta'(t)+\int_\Omega\theta^\beta\theta_t^2\dd x=-\frac12\int_\Omega
 \frac{\theta^{2\beta}\theta_x^2u_x}{v^2}\dd x
 -\int_\Omega\frac{\theta^{\beta+1}u_x\theta_t}{v}\dd x+\int_\Omega\frac{\theta^\beta
 (\mu u_x^2+|\mathbf{w}_x|^2+|\mathbf{b}_x|^2)\theta_t}{v}\dd x.
                                                               \label{eq:flux-t}
\end{align}
The boundary term from the heat flux vanishes by
\eqref{eq:bc-C}--\eqref{eq:bc-D}; in both half-line problems,
$\theta_x(0,t)=0$. Young's inequality, \eqref{eq:state}, and
$\norm{u_x(t)}_{L^\infty}\leq1+\norm{u_x(t)}_{L^\infty}^2$ yield
\begin{align}
 Y_\beta'(t)&+c\int_\Omega\theta^\beta\theta_t^2\dd x
 \leq C\norm{u_x(t)}_{L^\infty}^2Y_\beta(t)+CY_\beta(t)+C\int_\Omega\biggl\{
 \theta^{\beta+2}u_x^2
 +\theta^\beta(\mu u_x^2+|\mathbf{w}_x|^2+|\mathbf{b}_x|^2)^2\biggr\}\dd x.
                                                               \label{eq:Y}
\end{align}

For every fixed $\beta\geq0$, \eqref{eq:hot-Lp} gives
\[
 \int_\Omega\theta^{\beta+2}u_x^2\dd x
 \leq C\norm{u_x(t)}_{L^2}^2+C\norm{u_x(t)}_{L^\infty}^2.
\]
Integration in time, followed by \eqref{eq:prehi} and
\eqref{eq:gmax}, gives
\begin{equation}
 \int_0^T\!\int_\Omega
 \theta^{\beta+2}u_x^2\dd x\dd t\leq C.
                                                               \label{eq:t-ux}
\end{equation}

Split the remaining source term over $\{\theta\leq2\}$ and
$\{\theta>2\}$. Since $\theta^\beta\leq C$ on the first set,    \eqref{eq:trflux}, 
   \eqref{eq:long}, and  \eqref{eq:hot-Lp},    yield
\begin{align}
 &\int_\Omega\theta^\beta
 (\mu u_x^2+|\mathbf{w}_x|^2+|\mathbf{b}_x|^2)^2\,\dd x\notag\\
 &\quad\leq C\{\norm{u_x(t)}_{L^\infty}^2+\norm{\mathbf{w}_x(t)}_{L^\infty}^2
 +\norm{\mathbf{b}_x(t)}_{L^\infty}^2\}
 \int_\Omega(\mu u_x^2+|\mathbf{w}_x|^2+|\mathbf{b}_x|^2)\,\dd x\notag\\
 &\qquad+C\{\norm{u_x(t)}_{L^\infty}^2+\norm{\mathbf{w}_x(t)}_{L^\infty}^2
 +\norm{\mathbf{b}_x(t)}_{L^\infty}^2\}^2
 \int_{\{\theta>2\}}\theta^\beta\,\dd x\notag\\
 &\quad\leq C\{\norm{u_x(t)}_{L^\infty}^2+\norm{\mathbf{w}_x(t)}_{L^\infty}^2
 +\norm{\mathbf{b}_x(t)}_{L^\infty}^2\}+C\{\norm{u_x(t)}_{L^\infty}^2+\norm{\mathbf{w}_x(t)}_{L^\infty}^2
 +\norm{\mathbf{b}_x(t)}_{L^\infty}^2\}^2.
                                                        \label{eq:mechwt}
\end{align}
The right-hand side is integrable in time by \eqref{eq:gmax}.

To apply Gronwall's inequality, we also need $Y_\beta\in L^1(0,T)$.
On $\{\theta\leq4\}$,
$\theta^{2\beta}\leq C_\beta\theta^{-1}$, while on
$\{\theta>4\}$ we use \eqref{eq:htgrad}.
Thus \eqref{eq:state}, \eqref{eq:prehi}, and \eqref{eq:htgrad} give
\begin{align*}
 \int_0^T Y_\beta(t)\dd t
 &\leq C_\beta\int_0^T\!\int_{\{\theta\leq4\}}
 \frac{\theta_x^2}{v\theta}\dd x\dd t
 +C\int_0^T\!\int_{\{\theta>4\}}\theta^{2\beta}\theta_x^2\dd x\dd t
 \leq C.
\end{align*}
Hence
\begin{equation}
 \int_0^T Y_\beta(t)\dd t\leq C.              \label{eq:Y-L1}
\end{equation}

The initial assumptions, \eqref{eq:init}, and the one-dimensional
embedding give $Y_\beta(0)\leq C$. For $0\leq t\leq \min\{1,T\}$, it follows from
\eqref{eq:gmax}, \eqref{eq:Y}-\eqref{eq:mechwt}, and Gronwall's inequality that
\[
Y_\beta(t)\leq C.
\]
Next, for $1\leq t\leq T$, it follows from
\eqref{eq:gmax} and \eqref{eq:Y-L1} that
\[
\int_{t-1}^{t}
\left(Y_\beta(s)+1+\norm{u_x(s)}_{L^\infty}^{2}\right)
\dd s\leq C.
\]
Moreover, \eqref{eq:gmax}, \eqref{eq:t-ux},
and \eqref{eq:mechwt} give
\[
\int_{t-1}^{t}\!\int_\Omega
\left\{
\theta^{\beta+2}u_x^2
+\theta^\beta
\left(\mu u_x^2+|w_x|^2+|b_x|^2\right)^2
\right\}\dd x\dd s
\leq C.
\]
Applying the uniform Gronwall inequality to \eqref{eq:Y}
on $[t-1,t]$, we therefore obtain
\[
Y_\beta(t)\leq C,\qquad 1\leq t\leq T,
\]
which together with the estimate for
$0\leq t\leq\min\{1,T\}$ gives
\[
\sup_{0\leq t\leq T}Y_\beta(t)\leq C.
\]
Integrating \eqref{eq:Y} over $(0,T)$, we obtain after
using the above bound, \eqref{eq:gmax} and 
\eqref{eq:t-ux}-\eqref{eq:Y-L1} that
\begin{equation}\label{eq:t-wt}
	\sup_{0\leq t\leq T}Y_\beta(t)
	+\int_0^T\!\int_\Omega
	\theta^\beta\theta_t^2\dd x\dd t\leq C,
\end{equation}
where $C$ is independent of $T$.

\noindent\textit{Pointwise and second-order bounds.}\quad
For every fixed $\beta\geq0$, the fundamental theorem of calculus on
the components of $\{\theta>2\}$ and the Cauchy--Schwarz inequality
give
\begin{align*}
 \norm{\bigl(\theta(t)^{\beta+1}-2^{\beta+1}\bigr)_+}_{L^\infty}
 &\leq(\beta+1)
 \left(\int_{\{\theta>2\}}
 \frac{\theta^{2\beta}\theta_x^2}{v}\dd x\right)^{\frac12}
 \left(\int_{\{\theta>2\}}v\dd x\right)^{\frac12}\leq C.
\end{align*}
Here we used \eqref{eq:t-wt}, \eqref{eq:state}, and \eqref{eq:hot}.
This proves \eqref{eq:t-up}.
Combining \eqref{eq:t-wt} with \eqref{eq:state}, we obtain
\begin{equation}\label{eq:theta-x-sup}
 \sup_{0\leq t\leq T}\int_\Omega\theta_x^2\dd x\leq C.
\end{equation}
In addition,
\begin{equation}\label{eq:t-time}
 \int_0^T\!\int_\Omega\theta_t^2\dd x\dd t\leq C.
\end{equation}

Finally, solving \eqref{eq:temp} for $\theta_{xx}$ gives
\begin{align}
 \theta_{xx}={}&v\theta^{-\beta}\theta_t
 +\theta^{1-\beta}u_x
 -\theta^{-\beta}(\mu u_x^2+|\mathbf{w}_x|^2+|\mathbf{b}_x|^2)-\beta\theta^{-1}\theta_x^2
 +\frac{v_x}{v}\theta_x.                            \label{eq:t-xx}
\end{align}
By \eqref{eq:state}, \eqref{eq:t-up}, and \eqref{eq:t-xx},
\begin{align}
 \norm{\theta_{xx}(t)}_{L^2}^2
 &\leq C\int_\Omega\bigl\{\theta_t^2+u_x^2
 +(\mu u_x^2+|\mathbf{w}_x|^2+|\mathbf{b}_x|^2)^2+\theta_x^4+v_x^2\theta_x^2\bigr\}\dd x.
                                                        \label{eq:txxbd}
\end{align}
The one-dimensional interpolation inequality and Young's inequality
give, for every $\varepsilon>0$,
\begin{align}
 \int_0^T\!\int_\Omega(\theta_x^4+v_x^2\theta_x^2)\dd x\dd t&\leq\varepsilon\int_0^T\norm{\theta_{xx}(t)}_{L^2}^2\dd t
 +C_\varepsilon\int_0^T\bigl\{\norm{\theta_x(t)}_{L^2}^6
 +\norm{v_x(t)}_{L^2}^4\norm{\theta_x(t)}_{L^2}^2\bigr\}\dd t\notag\\
 &\quad\leq\varepsilon\int_0^T\norm{\theta_{xx}(t)}_{L^2}^2\dd t+C_\varepsilon,
                                                        \label{eq:txxrem}
\end{align}
where the final estimate in \eqref{eq:txxrem} follows from \eqref{eq:long},
\eqref{eq:theta-x-sup}, and \cref{l:vx}. Combining   \eqref{eq:long},  \eqref{eq:gmax},  \eqref{eq:mechwt},   \eqref{eq:theta-x-sup},  \eqref{eq:t-time},   \eqref{eq:txxbd},   \eqref{eq:txxrem}, 
and choosing $\varepsilon$ so that the $\theta_{xx}$ term in
\eqref{eq:txxrem} is absorbed into the left-hand side, we obtain
\eqref{eq:t-H2}.
\end{proof}

\section{Large-time behavior and proof of the main theorem}\label{sec:full}
We first collect the estimates required by the continuation criterion and
then prove convergence to the constant equilibrium.
\begin{proposition}\label{p:full}
Let $0\leq\alpha\leq\alpha_0$ and $T<T_*$. Then
\begin{align}
 \sup_{0\leq t\leq T}\bigl\{
 \norm{v(t)-1}_{H^1(\Omega)}^2+\norm{u(t)}_{H^1(\Omega)}^2
 +\norm{\theta(t)-1}_{H^1(\Omega)}^2+\norm{\mathbf{w}(t)}_{H^1(\Omega)}^2+\norm{\mathbf{b}(t)}_{H^1(\Omega)}^2\bigr\}\leq C,
                                                               \label{eq:fullH1}
\end{align}
and
\begin{align}
 \int_0^T\!\int_\Omega\bigl\{&
 v_x^2+\mu u_x^2+|\mathbf{w}_x|^2+|\mathbf{b}_x|^2+\theta_x^2\notag\\
 &+u_{xx}^2+|\mathbf{w}_{xx}|^2+|\mathbf{b}_{xx}|^2+\theta_{xx}^2\notag\\
 &+u_t^2+|\mathbf{w}_t|^2+|\mathbf{b}_t|^2+\theta_t^2\bigr\}
 \dd x\dd t\leq C.                                    \label{eq:fullds}
\end{align}
Here $C$ depends only on the fixed data and $\beta$, and is independent of
$T$ and $\alpha\in[0,\alpha_0]$.
\end{proposition}

\begin{proof}
Using \eqref{eq:ent-id},  \eqref{eq:state},    \eqref{eq:trflux},  lemma \ref{l:vx},  and propositions \ref{p:long}-\ref{p:thig},   we obtain
\eqref{eq:fullH1}.

It remains to estimate the second spatial derivatives of $\mathbf{w}$ and $\mathbf{b}$.
Since
\begin{equation}\label{eq:wxx-flux}
 \mathbf{w}_{xx}=v\left(\frac{\mathbf{w}_x}{v}\right)_x
 +v_x\frac{\mathbf{w}_x}{v}.
\end{equation}
The corresponding identity for $\mathbf{b}$ is
\begin{equation}\label{eq:bxx-flux}
 \mathbf{b}_{xx}=v\left(\frac{\mathbf{b}_x}{v}\right)_x
 +v_x\frac{\mathbf{b}_x}{v}.
\end{equation} By \eqref{eq:state}, \eqref{eq:trflux}, \eqref{eq:tr-max}, and  \cref{l:vx} 
\begin{equation}\label{eq:transverse-H2}
\begin{aligned}
 \int_0^T\!\int_\Omega(|\mathbf{w}_{xx}|^2+|\mathbf{b}_{xx}|^2)\dd x\dd t&\leq C\int_0^T\!\int_\Omega
 \left\{\left|\left(\frac{\mathbf{w}_x}{v}\right)_x\right|^2
 +\left|\left(\frac{\mathbf{b}_x}{v}\right)_x\right|^2\right\}\dd x\dd t\\&+C\sup_{0\leq t\leq T}\norm{v_x(t)}_{L^2}^2
 \int_0^T\{\norm{\mathbf{w}_x(t)}_{L^\infty}^2
 +\norm{\mathbf{b}_x(t)}_{L^\infty}^2\}\dd t\leq C.
\end{aligned}
\end{equation}
Combining \eqref{eq:transverse-H2} with
\cref{l:vx,p:cap,p:long,%
p:thig} and \eqref{eq:trflux} gives
\eqref{eq:fullds}.
\end{proof}

The estimates in Proposition~\ref{p:full}, the bounds
\eqref{eq:state}, and Proposition~\ref{p:thig} are uniform for
$T<T_*$. Suppose that $T_*<\infty$. Then
$v_t=u_x\in L^2(0,T_*;H^1)$ implies that $v-1$ has a limit in $H^1$
as $t\uparrow T_*$. For each diffusive variable, the bounds in
$L^2(0,T_*;H^2)\cap H^1(0,T_*;L^2)$ imply continuity into $H^1$ up to
$T_*$, with the prescribed boundary traces.
Indeed, on every finite interval,
\begin{equation}\label{eq:time-continuity}
 L^2(0,T;H^2)\cap H^1(0,T;L^2)
 \hookrightarrow C([0,T];H^1).
\end{equation}
The uniform positive state bounds pass to these terminal data.
Lemma~\ref{lem:01a} then restarts the solution beyond $T_*$,
a contradiction. Therefore $T_*=\infty$. Letting $T\to\infty$
proves \eqref{eq:mainvt}--\eqref{eq:mainds}.

\begin{proposition}\label{p:dec}
For each fixed $\alpha\in[0,\alpha_0]$, as $t\to\infty$,
\begin{equation}
 \norm{v_x(t)}_{L^2}+\norm{u_x(t)}_{L^2}+\norm{\theta_x(t)}_{L^2}
 +\norm{\mathbf{w}_x(t)}_{L^2}+\norm{\mathbf{b}_x(t)}_{L^2}\longrightarrow0.    \label{eq:grdec}
\end{equation}
Moreover, for every $2<p\leq\infty$,
\begin{align}
 \norm{v(t)-1}_{L^p(\Omega)}+\norm{u(t)}_{L^p(\Omega)}
 +\norm{\theta(t)-1}_{L^p(\Omega)}+\norm{\mathbf{w}(t)}_{L^p(\Omega)}+\norm{\mathbf{b}(t)}_{L^p(\Omega)}
 \longrightarrow 0.                                  \label{eq:Lp-dec}
\end{align}
 
\end{proposition}

\begin{proof}
By \eqref{eq:fullds}, the squares of all five $L^2$ norms
in \eqref{eq:grdec} belong to $L^1(0,\infty)$.
We compute their time derivatives directly; the same calculation is used
in \cite{LiLiang2016}. For $v$, the identity $v_{xt}=u_{xx}$ gives
\begin{equation}\label{eq:dec-vx}
 \frac{\dd}{\dd t}\int_\Omega v_x^2\dd x
 =2\int_\Omega v_xu_{xx}\dd x.
\end{equation}
The right-hand side belongs to $L^1(0,\infty)$ by the
Cauchy--Schwarz inequality and \eqref{eq:fullds}.

Let $f$ be any scalar component of $u$, $\mathbf{w}$, $\mathbf{b}$, or $\theta-1$.
Integration by parts and the corresponding Dirichlet or Neumann
boundary condition give, for almost every $t$,
\begin{equation}\label{eq:dec-fx}
 \frac{\dd}{\dd t}\int_\Omega f_x^2\dd x
 =-2\int_\Omega f_{xx}f_t\dd x.
\end{equation}
Again, the right-hand side belongs to $L^1(0,\infty)$ by
\eqref{eq:fullds}.
Hence the squares of all five derivative norms belong to
$L^1(0,\infty)\cap W^{1,1}(0,\infty)$.
Each has a limit at infinity, which must be zero by integrability. This
proves \eqref{eq:grdec}.

Let $g$ be any scalar component of $v-1,u,\theta-1,\mathbf{w},\mathbf{b}$.
By \eqref{eq:fullH1}, $\norm{g(t)}_{L^2}$ is uniformly bounded,
while \eqref{eq:grdec} gives
$\norm{g_x(t)}_{L^2}\to0$.
The one-dimensional Gagliardo--Nirenberg inequality
\eqref{eq:GN} yields
\[
 \norm{g(t)}_{L^p(\Omega)}
 \leq C_p\norm{g(t)}_{L^2}^{\frac12+\frac1p}
 \norm{g_x(t)}_{L^2}^{\frac12-\frac1p}\longrightarrow0
 \qquad(2<p\leq\infty).
\]
Summing over all scalar components proves \eqref{eq:Lp-dec}.

\end{proof}

\begin{proof}[Proof of Theorem~\ref{thm:main}]
The choice \eqref{eq:alpha-choice} and the bootstrap argument for $v$
establish \eqref{eq:state}. The estimates in Section~\ref{sec:high}
yield Proposition~\ref{p:full} and the upper bound for $\theta$.
The continuation argument based on Proposition~\ref{p:full} and
Lemma~\ref{lem:01a} gives global existence and uniqueness.
Proposition~\ref{p:dec} gives
\eqref{eq:maingr} for each fixed admissible $\alpha$. The constants in
\eqref{eq:mainvt}--\eqref{eq:mainds} are uniform for
$\alpha\in[0,\alpha_0]$: the primitive comparison and the absorption in
\eqref{eq:new-error} use fixed constants, while $\mu$ and its required state derivatives are
uniformly bounded on the state interval \eqref{eq:state}.
\end{proof}

Smallness of $\alpha$ is used only in the primitive comparison and in the
absorption of \eqref{eq:new-error}. No decay rate uniform in $\alpha$ is
asserted.

\appendix
\section{Justification of the energy estimates}\label{app:energy}
The energy identities used in Propositions~\ref{p:cap}--\ref{p:dec} are
justified by approximation in the class \eqref{eq:finite}. Smooth data
with homogeneous Dirichlet values
or with constant values near a Neumann endpoint are dense in the
corresponding $H^1$ space; truncation, a shrinking endpoint collar,
and mollification preserve the positive state bounds. The corresponding
difference estimates give convergence of the approximating solutions in
$C_tH^1$ and in $L^2_tH^2$ for the diffusive variables, and the
equations give convergence of their time derivatives in $L^2$.
For a component $f$ of $u,\mathbf{w},\mathbf{b},\theta-1$, time convolution preserving
the homogeneous boundary condition consequently gives
\begin{equation}\label{eq:energy-fx}
  \frac{\dd}{\dd t}\int_\Omega\frac{|f_x|^2}{v}\dd x
 =-2\int_\Omega\left(\frac{f_x}{v}\right)_xf_t\dd x
   -\int_\Omega\frac{u_x|f_x|^2}{v^2}\dd x.
\end{equation}
Indeed, $\left(\frac{f_x}{v}\right)_x\in L^2_{t,x}$ follows from the
finite-time regularity in \eqref{eq:finite}, while
$u_x\in L^2_tL^\infty_x$.
Thus formula \eqref{eq:energy-fx} does not require a boundary trace of
$f_t$.
For the unweighted norm, integration by parts at the regularized level
gives \eqref{eq:dec-fx} directly, and the convergence stated in the first
paragraph permits passage to the limit. For the heat time-derivative test
leading to \eqref{eq:flux-t}, apply time convolution to
$\frac{\theta^{\beta+1}-1}{\beta+1}$. The regularity in
\eqref{eq:finite} and the finite-time positive state bounds justify passage
to the limit.
Lipschitz approximations of the positive part justify the temperature
truncations, and spatial cutoffs remove the boundary at infinity.
The half-line boundary fluxes vanish by $u=\mathbf{w}=0$, $\theta_x=0$,
and either $\mathbf{b}=0$ or $\mathbf{b}_x=0$. Identity \eqref{eq:lj09}
is justified by the same approximation used for \eqref{eq:energy-fx}: on each finite positive
state interval,
$q_t=(\mu u_x/v)_x\in L^2_{t,x}$ and $u\in C_tH^1_x$.
The stress in \eqref{eq:sigma} satisfies
$\sigma+1\in L^2(0,T;H^1(\Omega))$, so the  identities in
Lemma~\ref{l:Kaz} are well defined. For the elliptic equation,
$v^{-1}\in W^{1,2}_tL^\infty_x$ on finite intervals; the detailed
differentiation argument is included in Lemma~\ref{l:res}.


\begin{thebibliography}{99}

\bibitem[Cabannes(1970)]{Cabannes1970}
H.~Cabannes.
\newblock \emph{Theoretical Magnetofluiddynamics}.
\newblock Academic Press, New York, (1970).

\bibitem[Cao et~al.(2021)Cao, Peng, and Sun]{CaoPengSun2021}
Y.~Cao, Y.~Peng, and Y.~Sun.
\newblock Global existence of strong solutions to {MHD} with density-depending
viscosity and temperature-depending heat-conductivity in unbounded domains.
\newblock \emph{J. Math. Phys.}, 62(1):011508, (2021).
\newblock \doi{10.1063/5.0039238}.

\bibitem[Chen and Wang(2002)]{ChenWang2002}
G.-Q.~Chen and D.~Wang.
\newblock Global solutions of nonlinear magnetohydrodynamics with large initial data.
\newblock \emph{J. Differ. Equ.}, 182(2):344--376, (2002).
\newblock \doi{10.1006/jdeq.2001.4111}.

\bibitem[Dong and Guo(2025)]{DongGuo2025}
W.~Dong and Z.~Guo.
\newblock Cauchy problem for the {Navier--Stokes} equations with
temperature-dependent transport coefficients and large data.
\newblock \emph{Math. Models Methods Appl. Sci.}, 35(9):1889--1932, (2025).
\newblock \doi{10.1142/S0218202525500319}.

\bibitem[Hu and Ju(2015)]{HuJu2015}
Y.~Hu and Q.~Ju.
\newblock Global large solutions of magnetohydrodynamics with
temperature-dependent heat conductivity.
\newblock \emph{Z. Angew. Math. Phys.}, 66(3):865--889, (2015).
\newblock \doi{10.1007/s00033-014-0446-1}.

\bibitem[Huang et~al.(2019)Huang, Shi, and Sun]{HuangShiSun2019}
B.~Huang, X.~Shi, and Y.~Sun.
\newblock Global strong solutions to magnetohydrodynamics with density-dependent
viscosity and degenerate heat-conductivity.
\newblock \emph{Nonlinearity}, 32(11):4395--4412, (2019).
\newblock \doi{10.1088/1361-6544/ab3059}.

\bibitem[Huang et~al.(2021)Huang, Shi, and Sun]{HuangShiSun2021}
B.~Huang, X.~Shi, and Y.~Sun.
\newblock Large-time behavior of magnetohydrodynamics with
temperature-dependent heat-conductivity.
\newblock \emph{J. Math. Fluid Mech.}, 23:67, (2021).
\newblock \doi{10.1007/s00021-021-00594-y}.

\bibitem[Jeffrey and Taniuti(1964)]{JeffreyTaniuti1964}
A.~Jeffrey and T.~Taniuti.
\newblock \emph{Non-Linear Wave Propagation: With Applications to Physics and
Magnetohydrodynamics}.
\newblock Academic Press, New York, (1964).

\bibitem[Jiang(1999)]{Jiang1999}
S.~Jiang.
\newblock Large-time behavior of solutions to the equations of a one-dimensional
viscous polytropic ideal gas in unbounded domains.
\newblock \emph{Commun. Math. Phys.}, 200(1):181--193, (1999).
\newblock \doi{10.1007/s002200050526}.

\bibitem[Jiang(2002)]{Jiang2002}
S.~Jiang.
\newblock Remarks on the asymptotic behaviour of solutions to the compressible
{Navier--Stokes} equations in the half-line.
\newblock \emph{Proc. R. Soc. Edinb. Sect. A}, 132:627--638, (2002).
\newblock \doi{10.1017/S0308210500001815}.

\bibitem[Kazhikhov(1987)]{Kazhikhov1987}
A.~V.~Kazhikhov.
\newblock A priori estimates for the solutions of equations of magnetic gas dynamics.
\newblock In \emph{Boundary Value Problems for Equations of Mathematical Physics},
Krasnoyarsk, (1987).
\newblock In Russian.

\bibitem[Landau and Lifshitz(1984)]{LandauLifshitz1984}
L.~D.~Landau and E.~M.~Lifshitz.
\newblock \emph{Electrodynamics of Continuous Media}.
\newblock 2nd edn., Pergamon, New York, (1984).

\bibitem[Li et~al.(2026)Li, Liao, and Wan]{LiLiaoWan2026}
F.~Li, Y.~Liao, and L.~Wan.
\newblock Large-time behavior of global solutions to the Cauchy problem of the
planar compressible magnetohydrodynamic equations.
\newblock \emph{J. Math. Fluid Mech.}, 28:52, (2026).
\newblock \doi{10.1007/s00021-026-01037-2}.

\bibitem[Li and Liang(2016)]{LiLiang2016}
J.~Li and Z.~Liang.
\newblock Some uniform estimates and large-time behavior of solutions to
one-dimensional compressible {Navier--Stokes} system in unbounded domains with
large data.
\newblock \emph{Arch. Ration. Mech. Anal.}, 220(3):1195--1208, (2016).
\newblock \doi{10.1007/s00205-015-0952-0}.

\bibitem[Li and Liu(2026)]{LiLiu2026}
J.~Li and Z.~Liu.
\newblock Large-time behavior for one-dimensional planar compressible
magnetohydrodynamic equations on unbounded domains.
\newblock arXiv:2609.20126v2, (2026).
\newblock \url{https://arxiv.org/abs/2609.20126v2}.

\bibitem[Li and Xu(2026)]{LiXu2026}
K.~Li and X.~Xu.
\newblock Large-time behavior of solutions to compressible {Navier--Stokes}
system in unbounded domains with degenerate heat-conductivity and large data.
\newblock \emph{J. Differ. Equ.}, 465:114204, (2026).
\newblock \doi{10.1016/j.jde.2026.114204}.

\bibitem[Li and Shang(2019)]{LiShang2019}
W.~Li and Z.~Shang.
\newblock Global large solutions to the planar magnetohydrodynamics equations
with constant heat conductivity.
\newblock \emph{Bound. Value Probl.}, 2019:162, (2019).
\newblock \doi{10.1186/s13661-019-1275-0}.

\bibitem[L{\"u} et~al.(2021)L{\"u}, Shi, and Xiong]{LuShiXiong2021}
B.~L{\"u}, X.~Shi, and C.~Xiong.
\newblock Global existence of strong solutions to the planar compressible
magnetohydrodynamic equations with large initial data in unbounded domains.
\newblock \emph{Commun. Math. Sci.}, 19(6):1655--1671, (2021).
\newblock \doi{10.4310/CMS.2021.v19.n6.a9}.

\bibitem[Shang(2022)]{Shang2022}
Z.~Shang.
\newblock Global large solutions to the Cauchy problem of planar
magnetohydrodynamics equations with temperature-dependent coefficients.
\newblock \emph{J. Dyn. Control Syst.}, 28(1):163--205, (2022).
\newblock \doi{10.1007/s10883-020-09526-x}.

\bibitem[Song and Zhao(2025)]{SongZhao2025}
D.~Song and X.~Zhao.
\newblock Large time behavior of strong solution to the magnetohydrodynamics
system with temperature-dependent viscosity, heat-conductivity, and resistivity.
\newblock \emph{Electron. Res. Arch.}, 33(2):938--972, (2025).
\newblock \doi{10.3934/era.2025043}.

\bibitem[Tong et~al.(2026)Tong, Wang, and Zhang]{TongWangZhang2026}
M.~Tong, X.~Wang, and R.~Zhang.
\newblock Global strong solutions to the planar compressible
magnetohydrodynamic equations with degenerate heat-conductivity in the half-line.
\newblock \emph{Acta Math. Sci. Ser. B}, 46(1):189--208, (2026).
\newblock \doi{10.1007/s10473-026-0112-7}.

\bibitem[Wang(2003)]{Wang2003}
D.~Wang.
\newblock Large solutions to the initial-boundary value problem for planar
magnetohydrodynamics.
\newblock \emph{SIAM J. Appl. Math.}, 63(4):1424--1441, (2003).
\newblock \doi{10.1137/S0036139902409284}.

\bibitem[Wang and Zhang(2026)]{WangZhang2026}
H.~Wang and R.~Zhang.
\newblock Initial and initial boundary value problems for the
{Navier--Stokes} equations with temperature-dependent transport coefficients
and large data.
\newblock arXiv:2609.24435, (2026).
\newblock \url{https://arxiv.org/abs/2609.24435}.
 

 

\end{thebibliography}
\end{document}